\documentclass[11pt]{article}
\usepackage{amsmath,amssymb,amsthm,mathtools,diagrams,theoremref,a4wide}

\DeclareMathOperator{\Hom}{Hom}

\DeclareMathOperator{\id}{id}

\DeclareMathOperator{\SO}{SO}

\DeclareMathOperator{\Sym}{Sym}

\DeclareMathOperator{\Aut}{Aut}
\DeclareMathOperator{\Out}{Out}
\DeclareMathOperator{\Res}{Res}
\DeclareMathOperator{\Ind}{Ind}

\DeclareMathOperator{\Mod}{-Mod}

\DeclareMathOperator{\Ext}{Ext}
\DeclareMathOperator{\Irr}{Irr}

\DeclareMathOperator{\FI}{FI}

\DeclareMathOperator{\Conf}{Conf}
\DeclareMathOperator{\Frob}{Frob}
\DeclareMathOperator{\UConf}{UConf}
\DeclareMathOperator{\FR}{FR}
\DeclareMathOperator{\gr}{gr}

\newtheorem{thm}{Theorem}[section]
\newtheorem{prop}[thm]{Proposition}

\newtheorem{cor}[thm]{Corollary}
\begin{document}

\title{\(\FI_G\)-modules, orbit configuration spaces,\\ and complex reflection groups}
\author{Kevin Casto}
\maketitle

\begin{abstract}
The category \(\FI_G\) was first defined and explored in \cite{SS2}. Here, we develop more of the machinery of \(\FI_G\)-modules and find numerous examples to apply it to, extending the work of \cite{CEF} and \cite{Wi}. 
In particular we develop a notion of character polynomials for \(\FI_G\)-modules with \(G\) finite, a notion of representation stability which we call \emph{\(K_0\)-stability} even when \(G\) is infinite virtually polycyclic, and apply the notion of \emph{finite presentation degree} when \(G\) is a general infinite group. We use this to analyze numerous families of \((G^n \rtimes S_n)\)-modules, such as:
\begin{itemize}
\item the cohomology and homotopy groups of orbit configuration spaces
\item the diagonal coinvariant algebra of complex reflection groups
\item the homology of affine pure braid groups of type \(\widetilde{A}_n\) and \(\widetilde{C}_n\)
\item the cohomology of Fouxe-Rabinowitsch groups
\end{itemize}
and many more examples.
\end{abstract}

\section{Introduction}
In \cite{Wi}, Wilson generalized the work of Church-Ellenberg-Farb \cite{CEF} on FI-modules to the setting of \(\FI_\mathcal{W}\)-modules, proving strong results about these and finding a wide range of interesting applications. In \cite{SS2}, Sam-Snowden substantially generalized Wilson's work to that of \(\FI_G\)-modules, for \(G\) a virtually polycyclic group, where Wilson's \(\FI_{\text{BC}} = \FI_{\mathbb{Z}/2\mathbb{Z}}\). Recall that a group is polycyclic if it has a subnormal series with cyclic factors, and virtually polycyclic if it has a polycyclic subgroup of finite index.

Sam-Snowden proved the Noetherian property for \(\FI_G\), and Gan-Li \cite{GL3} extended this to a proof of actual representation stability in the sense of \cite{CF}. However, few examples of \(\FI_G\)-modules have been discussed. Here, we extend Wilson's work to exhibit a wide class of examples of \(\FI_G\)-modules. For any virtually polycyclic group \(G\), we give applications to the following sequences of representations of \(W_n = G^n \rtimes S_n\) (see \S3.1 and \S7 for definitions):
\begin{enumerate}
\item The homology \(H_i(\Conf_n^G(M); \mathbb{C})\) of the orbit configuration space \(\Conf_n^G(M)\) associated to the action of \(G\) on an open manifold \(M\), and the cohomology \(H^i(\Conf_n^G(M); \mathbb{C})\) in the case where \(M\) need not be open but \(G\) is finite.
\item The rational homotopy groups \(\pi_i(\Conf_n(M))\) of an open manifold \(M\) of dimension \(\ge 3\) with \(G = \pi_1(M)\) virtually polycyclic, or the dual rational homotopy groups \(\Hom(\pi_i(\Conf_n(M)), \mathbb{Q})\) where \(M\) need not be open but \(G\) is finite.
\item The cohomology \(H^i(\FR(G^{\ast n}); \mathbb{C})\) of the Fouxe-Rabinovitch group of \(G\), where each \(H^i(G; \mathbb{Q})\) is finite-dimensional.
\end{enumerate}
For the specific case \(G = \mathbb{Z}/d\mathbb{Z}\), the resulting automorphism groups \(W_n = \left(\mathbb{Z}/d\mathbb{Z}\right)^n \wr S_n\) are the so-called main series of complex reflection groups. These directly generalize the Weyl groups of type \(BC_n\) of Wilson's paper, for which \(d = 2\). We give applications to the following sequence of representations of \(W_n\) (see \S4.1, \S4.2, \S4.3 for definitions):
\begin{enumerate}
\setcounter{enumi}{3}
\item The cohomology \(H^i(P(d,n); \mathbb{C})\) of the pure monomial braid group.
\item The graded pieces \(\gr P(d,n)^i\) of the associated graded Lie algebra of the pure monomial braid group.
\item The graded pieces \(\mathcal{C}_J^{(r)}(n)\) of the complex-reflection diagonal coinvariant algebra.
\end{enumerate}
Our work implies the following results about these sequences. Recall that the irreducible representations of \(W_n\) are parameterized by the partition-valued functions \(\underline \lambda\) on the irreducible representations of \(G\) with \(\|\underline \lambda\| = n\). Denote the  irreducible representation of \(W_n\) associated to \(\underline \lambda\) by \(L(\underline \lambda)\). Given partitions \(\lambda_1, \dots, \lambda_k\) with \(|\lambda_1| + \dots + |\lambda_k| = n\), and irreducible representations \(\chi_1, \dots, \chi_k\) of \(G\), we will write \(\underline\lambda = ( {\lambda_1}_{\chi_1}, \dots, {\lambda_2}_{\chi_2})\) for the partition-valued function \(\underline\lambda\) with \(\underline\lambda(\chi_i) = \lambda_i\).

Let \(c(G)\) be the set of conjugacy classes of \(G\). Define a \emph{character polynomial} to be a polynomial in \(c(G)\)-labeled cycle-counting functions. This is a class function on \(W_n\). (We will define this and other terminology more precisely in \S2.2).
\begin{thm}[\bfseries Polynomiality of characters and representation stability]
Suppose \(G\) is finite, and let \(\{V_n\}\) be any of the sequences 1-6 above. Then:\begin{enumerate}
\item The characters of \(V_n\) are given by a single character polynomial for all \(n \gg 0\).
\item The multiplicity of each irreducible \(W_n\)-representation in \(V_n\) is eventually independent of \(n\), and \(\dim V_n\) is eventually polynomial in \(n\). 

Moreover, for sequences 3 and 4, we can explicity bound when this stabilization occurs.
\end{enumerate}
\end{thm}
For example, for sequence 4, once \(n \ge 2\) we obtain:
\[H^1(P(d,n); \mathbb{C}) = L((n)_{\chi_0}) \oplus L((n-1,1)_{\chi_0}) \oplus \bigoplus_{\chi \in \Irr(G)} L((n-2)_{\chi_0}, (2)_\chi) \]
In particular, our results on \(H^i(\FR(\Gamma^{\ast n}); \mathbb{C})\) answer a question of Wilson, posed in \cite[\S 7.1.1]{Wi}, of how to generalize her Thm 7.3 to a general Fouxe-Rabinovitch group. Our results on \(\pi_i(\Conf_n(M))\) both confirm and correct a sketch given by Kupers-Miller \cite{KM} about the homotopy groups of configuration spaces of non-simply-connected manifolds, and answer their question about obtaining a form of representation stability for finite \(G\).
\subsection*{Infinite \(G\): \(K_0\)-stability and finite presentation degree}
When \(G\) is infinite virtually polycyclic, we still obtain a form of representation stability. For a ring \(R\), recall that the Grothendieck group \(G_0(R)\) is the free abelian group on finitely-generated \(R\)-modules, modulo the relations that the center of a short exact sequence is equal to the sum of the left and right terms. \(K_0(R)\) is defined similarly, but only for projective \(R\)-modules; these two groups coincide when \(R = k[G]\) for \(G\) virtually polycyclic. Say that a sequence \(V_n\) of \(W_n\)-representations satisfies \emph{\(K_0\)-stability} if the decomposition of \([V_n]\) in the \emph{Grothendieck group} \(K_0(k[W_n])\) is eventually independent of \(n\). We apply this notion to 1-3 above, and to the following two more examples:
\begin{enumerate} \setcounter{enumi}{6}
\item The homology \(H_i(P_{\widetilde{A}_{n-1}}; \mathbb{Q})\) of the affine pure braid group of type \(\widetilde{A}_{n-1}\), for which \(G = \mathbb{Z}\).
\item The homology \(H_i(P_{\widetilde{C}_{n}}; \mathbb{Q})\) of the affine pure braid group of type \(\widetilde{C}_{n}\), for which \(G = \mathbb{Z} \rtimes \mathbb{Z}/2\).
\end{enumerate}
\begin{thm}[\bfseries \(K_0\)-stability]
Let \(\{V_n\}\) be any of the sequences above (1-3,7,8) for which \(G\) is infinite virtually polycyclic. Then \(V_n\) satisfies \(K_0\)-stability.
\end{thm}
For example, if \(U = k[\mathbb{Z}] = \Ind_{\mathbb{Z}/2}^{\mathbb{Z} \rtimes \mathbb{Z}/2} k\), then in the notation of \S2.3,
\[ [H_1(P_{\widetilde{C}_{n}}; \mathbb{Q})] = [L((n)_U)],\;\; \forall n \ge 1. \]
Note that in general we certainly lose information in passing to \(K_0\). For instance, in the above example, we actually explicitly determine in \S6.(\ref{affC}) that \(H_1(P_{\widetilde{C}_{n}}; \mathbb{Q})\) is the direct sum of three induced representations, two of which are killed in passing to \(K_0\), and the third giving the term \([L((n)_U)]\) above.

This is the first representation stability result obtained for infinite discrete groups (recall that \cite{CF} obtained representation stability results for Lie groups, whose representation theory is much more tractable). 

Furthermore, for a totally general group \(G\), we still obtain a weak form of finite generation. Say that an \(\FI_G\)-module is \emph{presented in finite degree} if it is finitely-generated and if the relations for the generators also live in finite degree. Ramos \cite{Ra2} and Li \cite{Li} have proven that the category of \(\FI_G\)-modules that are presented in finite degree is abelian. This allows us to carry over arguments about spectral sequences to conclude that cohomology is generated in finite degree. In particular, we have:
\begin{thm}[\bfseries Finite presentation degree]
The cohomology \(H^i(\Conf^G_n(M); \mathbb{Q})\) and the dual rational homotopy groups \(\Hom(\pi_i( \Conf_n(M)); \mathbb{Q})\) are \(\FI_G\)-modules presented in finite degree.
\end{thm}

\subsection*{Application: arithmetic statistics for Gauss sums}
In the sequel \cite{Ca2}, we apply the theory of representation stability for \(\FI_G\)-modules to obtain results about weighted point-counts for polynomials over finite fields. For example, we are able to obtain results such as the following. Let \(\chi\) be a character of \(\mathbb{Z}/(q-1)\). Define the character polynomial \(X_i^\chi = \sum_{g \in G} \chi(g) X^g_i\). Then
\begin{gather} \begin{aligned} \label{X2_pcount}
\lim_{n \to \infty} &\sum_{f \in \UConf_n(\mathbb{F}_q^*)} \sum_{\deg(p) = 2} \chi(\text{root}(p))  \\
&= \sum_i (-1)^i \frac{\langle X_2^{\overline \chi}, H^i(\Conf^{\mathbb{Z}/(q-1)}(\mathbb{C}^*); \mathbb{Q}(\zeta_{q-1}))\rangle}{q^i} = \frac{1}{2q} + \frac{2}{q^2} + \cdots
\end{aligned} \end{gather}
That is, the average value of the Gauss sum obtained by applying \(\chi\) to the \emph{quadratic} factors of \(f\), across all \(f \in \UConf_n(\mathbb{F}^*_q)\), is equal to the series on the right obtained by looking at the inner product of the character polynomial \(X_2^{\overline \chi}\) with \(H^i(\Conf^{\mathbb{Z}/(q-1)}(\mathbb{C}^*); \mathbb{Q}(\zeta_{q-1}))\). The fact that the sum on the right is independent of \(n\) for large \(n\) is precisely a consequence of representation stability for the \(\FI_G\)-module \(H^i(\Conf^{\mathbb{Z}/(q-1)}(\mathbb{C}^*))\).

\subsection*{Acknowledgements}
I would like to thank Nir Gadish, Weiyan Chen, Jesse Wolfson, John Wiltshire-Gordon, Eric Ramos, and Jeremy Miller for helpful conversations. Also, I am very grateful to my advisor, Benson Farb, for all his guidance throughout the process of working on, writing, and revising my first paper.

\section{\(\FI_G\)-modules and their properties}
If \(G\) is a group, and \(R\) and \(S\) sets, define a \emph{\(G\)-map} \((a, (g_i)): R \to S\) to be a pair \(a: R \to S\) and \((g_i) \in G^R\). If \((b, (h_j)): S \to T\) is another \(G\)-map, their composition is \((b \circ a, (g_i \cdot h_{a(i)}))\). Let \(\FI_G\) be the category with objects finite sets and morphisms \(G\)-maps with the function \(a\) injective. This is clearly equivalent to the full subcategory with objects the sets \([n] = \{1, \dots, n\}\). Note that the automorphism group of \([n]\) is
\[W_n := G \wr S_n = G^n \rtimes S_n\]
An \emph{\(\FI_G\)-module over \(k\)} is just a functor \(V: \FI_G \to k\Mod\); when \(k\) is clear, we simply write \(\FI_G\)-modules. These form a category, called \(\FI_G\Mod\). Thus an \(\FI_G\)-module \(V\) is a sequence of \(W_n\)-representations \(V_n\), with maps \(V_n \to V_{n+1}\) satisfying certain coherency conditions. Say that \(V\) is \emph{finitely generated} if there is a finite set of elements \(v_1, \dots, v_n \in V\) such that the smallest \(\FI_G\)-submodule containing the \(v_i\) is all of \(V\).

For \(m \ge 0\), define the ``free'' \(\FI_G\)-module \(M(m)\) by setting
\[M(m)_n = \begin{cases} 0 & n < m \\ k[\Hom_{\FI_G}([m], [n])] & n \ge m \end{cases}\]
Recall that for an \(\FI_G\)-module \(V\) and an element \(v \in V_m\), we can also characterize the submodule generated by \(v\) as the image of the map
\[ M(m) \to V, \;\; f \in \Hom_{\FI_G}([m], [n]) \mapsto f_* v \]
We can thus characterize finitely generated \(\FI_G\)-modules as those \(V\) that admit a surjection \(\bigoplus_{i=1}^N M(m_i) \twoheadrightarrow V\).

\subsection{Noetherianity and representation stability}

Recall that a group is \emph{polycyclic} if it has a composition series with cyclic factors, and is \emph{virtually polycyclic} if it has a polycyclic subgroup of finite index. Virtually polycyclic groups are of interest, among other things, because they are the only known groups to have Noetherian group rings, and are conjectured to be the only such groups. In \cite[Cor 1.2.2]{SS2}, Sam-Snowden proved that if \(G\) is a virtually polycyclic group, then \(\FI_G\Mod\) is Noetherian over any Noetherian ring \(k\). That is, they proved that any finitely generated \(\FI_G\)-module has all its submodules finitely generated. The crucial property used was that the group ring \(k[G]\) is Noetherian.

For FI-modules, the most important consequence of finite generation is representation stability, and for \(G\) finite, Gan-Li proved in \cite[Thm 1.10]{GL3} that this holds for \(\FI_G\) as well. Technically there are three parts to representation stability according to the definition given in \cite{CF}, but the first two parts (``surjectivity'' and ``injectivity'') follow straightforwardly from the definition of being finitely generated. It is the third part, ``multiplicity stability'', that is really the most interesting, and which we will now describe.

Take \(G\) finite and \(k\) a \emph{splitting field} of characteristic 0 for \(G\), that is, a field over which all its irreducible representations over \(\mathbb{C}\) are defined. Let \(\Irr(G)\) denote the set of isomorphism classes of irreducible representations of \(G\). Let \(\underline \lambda\) be a partition-valued function on \(\Irr(G)\). Put \(|\underline\lambda| = (|\underline\lambda(\chi_1)|, \dots, |\underline\lambda(\chi_r)|)\), and \(\|\underline\lambda\|\) for the norm of this partition. Then if \(\|\underline\lambda\| = n\), there is an associated irreducible representation of \(W_n\):
\[ L(\underline \lambda) = \Ind^{W_n}_{W_{|\underline\lambda|}} \left(M_1^{\otimes \underline\lambda(\chi_1)} \otimes E(\underline\lambda(\chi_1))\right) \otimes \cdots \otimes \left(M_r^{\otimes \underline\lambda(\chi_r)} \otimes E(\underline\lambda(\chi_1))\right) \]
where \(W_\mu = W_{\mu(1)} \times \cdots W_{\mu(l)}\),
and that these comprise all the irreducible representations of \(W_n\) up to isomorphism. Extend \(\underline{\lambda}\) to \(n \ge m + \underline{\lambda}(\chi_0)_1\) as follows:
\[
\underline{\lambda}[n](\chi) = \begin{cases} \left(n - |\underline{\lambda}|, \underline{\lambda}(\chi_0)\right) & \text{if } \chi = \chi_0 \\ \underline{\lambda}(\chi) & \text{otherwise} \end{cases}
\]
Writing \(L(\underline{\lambda})_n\) for the irreducible representation corresponding to \(\underline{\lambda}[n]\), multiplicity stability for an \(\FI_G\)-module \(V\) says that the decomposition into irreducibles has multiplicities independent of \(n\) for large \(n\):
\[
V_n = \bigoplus_{\underline{\lambda}} L(\underline{\lambda})_n^{\oplus c(\underline{\lambda})} \,\text{ for all } n \ge N
\]
where we call \(N\) the \emph{stability degree} of \(V\). In particular, when \(G\) is trivial, we recover (uniform) multiplicity stability for FI-modules in the sense of \cite[Defn 2.6]{CF}, and when \(G = \mathbb{Z}/2\mathbb{Z}\), we recover (uniform) multiplicity stability for \(\FI_{\text{BC}}\)-modules in the sense of \cite[Defn 2.6]{Wi}.

\subsection{Projective resolutions and character polynomials}
For \(G\) finite, \cite{SS2} and \cite{GL3} actually obtain a deeper structural result that implies representation stability, and this result has other important consequences for us. To state it, define a \emph{torsion} \(\FI_G\)-module to be one with \(V_n \ne 0\) for only finitely many \(n\), and say that an \(\FI_G\)-module is projective if it is a projective object in the category \(\FI_G\Mod\). Gan-Li's result can then be stated as follows.

\begin{prop}[{\cite[Thm 1.6]{GL3}}]
For any finitely generated \(\FI_G\)-module \(V\), with \(G\) finite, there is a finite resolution of \(\FI_G\) modules
\[0 \to V \to T^1 \oplus P^1 \to T^2 \oplus P^2 \to \dots \to T^n \oplus P^k \to 0\]
with each \(T^i\) torsion and each \(P^i\) projective.
\end{prop}

In particular, this resolution's existence means that for \(n \gg 0\) there is a resolution of \(W_n\)-representations
\begin{equation} \label{fin_res}
0 \to V_n \to P^1_n \to \dots \to P^k_n \to 0
\end{equation}

This are powerful because, as in \cite{CEF}, we have strong control over the structure of projective \(\FI_G\)-modules. Namely, let \(\Res: \FI_G \to W_i\) be the restriction to a single group. This functor has a left adjoint \(\Ind^{\FI_G}: W_i \to \FI_G\) given by
\[\Ind^{\FI_G}(V)_{i+j} = \Ind_{W_i \times W_j}^{W_{i+j}} V \boxtimes k\]
Then \(\Ind^{\FI_G}(V)\) is a projective \(\FI_G\)-module whenever \(V\) is a projective \(W_i\)-module, and tom Dieck \cite[Prop 11.18]{tomD} proved that any projective \(\FI_G\)-module is of this form (in fact, for a large class of category representations). Following Ramos \cite{Ra1}, define a \emph{relatively projective} module to be a direct sum of any such induced modules, even if the \(W_i\) are not projective \(W_n\)-representations. Finitely-generated (relatively) projective \(\FI_G\)-modules thus have a compact description, as direct sums of induced modules from a finite list of \(W_n\)-representations, even for infinite \(G\).

In fact, Ramos \cite{Ra1} has found an effective bound for when \(n\) occurs in (\ref{fin_res}). To describe it, we need to refine our notion of finite generation. As we said, finitely generated modules are those that admit a surjection \(\bigoplus_{i=1}^N M(d_i) \twoheadrightarrow V\). Say that such a \(V\) is \emph{generated in degree} \(\le d = \max_i \{d_i\}\). Next, suppose there is an exact sequence
\[ 0 \to K \to M \to V \to 0\]
with \(M\) relatively projective and \(K\) generated in degree \(\le r\). Then say that \(V\) is has \emph{related in degree} \(\le r\).

Ramos \cite[Thm C]{Ra1} then says that if \(V\) is a finitely-generated \(\FI_G\)-module generated in degree \(d\) and related in degree \(r\), then the above resolution (\ref{fin_res}) holds whenever \(n \ge r + \min\{r, d\}\). Notice that if \(V\) is already relatively projective, then \(r = 0\) and this is sharp.

In particular, when \(G\) is finite, the resolution (\ref{fin_res}) implies representation stability: as Gan-Li verify in \cite[Thm 1.10]{GL3}, the individual projective modules \(\Ind^{\FI_G}(V_m)\) satisfy representation stability, with stability degree \(\le 2m\). By semisimplicity of each \(k[G_n]\), \(V\) therefore satisfies representation stability as well. Furthermore, this resolution provides a quick proof of the existence of character polynomials for \(\FI_G\), as follows.

Recall that for a \emph{character polynomial} for \(S_n\) is an element of \(\mathbb{Q}[X_1, X_2, \dots]\), which we think of as a class function on \(S_n\), where \(X_i\) counts the number of \(i\)-cycles of a permutation. \cite[Thm 3.3.4]{CEF} prove that any finitely generated \(\FI\)-module is eventually given by a single character polynomial. We generalize this and Wilson's \cite[Thm 5.15]{Wi} result for \(G = \mathbb{Z}/2\) as follows.
\begin{thm}[\bfseries Character polynomials for \(\FI_G\)]
\thlabel{charpoly}
Let \(G\) be a finite group and \(k\) a splitting field for \(G\) of characteristic 0. If \(V\) is a finitely generated \(\FI_G\)-module over \(k\), generated in degree \(m\) and related in degree \(r\), then there is a polynomial
\[P_V \in k\left[\{X_i^C \mid i \ge 1, C \text{ is a conjugacy class of } G\}\right]\]
of degree \(\le m\), where \(X_i^C\) is the class function counting the number of \(C\)-labeled \(i\)-cycles in the conjugacy class of \(g\), so that for all \(n \ge r + \min(m,r)\)
\[\chi_{V_n}(g) = P_V(g).\]
\end{thm}
\begin{proof}
If \(V\) is a representation of \(W_m\), we can explicitly compute the character of the projective \(\FI_G\)-module \(\Ind^{\FI_G}(V)\). This calculation is done in \cite[Lem 5.14]{Wi} for \(G = \mathbb{Z}/2\), and her proof applies essentially verbatim, but enough small details are different that it is easier to just give the adapted proof than to describe the necessary changes.

The character of the induced representation \(\Ind^{\FI_G}(V)_n = \Ind_{W_m \times W_{n-m}}^{W_n} V \boxtimes k\) is
\begin{align*}
\chi_{\Ind^{\FI_G}(V)_n}(w) &= \sum_{\substack{\{\text{cosets } C\, \mid\, w \cdot C = C\} \\ \text{any } s \in C}} \chi_{V \boxtimes k}(s^{-1} w s) \\
&= \sum_{\substack{\{\text{cosets } C\, \mid\, w \cdot C = C\} \\ \text{any } s \in C}} \chi_V(p_m(s^{-1} w s))
\end{align*}
where \(p_m: W_m \times W_{n-m} \to W_m\) is the projection, and where the sum is over all cosets in \(W_n / (W_m \times W_{n-m})\) that are stabilized by \(w\), equivalently, those cosets \(C\) such that \(s^{-1}ws \in W_m \times W_{n-m}\) for any \(s \in C\).

An element \(w \in W_n\) can be conjugated in \(W_m \times W_{n-m}\) precisely when its \(c(G)\)-labeled cycles can be split into a set of cycles of total length \(m\), and a set of cycles of total length \((n-m)\). If we fixed a labeled partition \(\underline \lambda\) of \(m\), then the cycles of \(w\) can be factored into an element \(w_m\) of labeled cycle type \(\underline \lambda\) and its complement \(w_{n-m}\) in the following number of ways (possibly 0):
\[
\prod_{C \in c(G)} \binom{X^C_1}{n_1(\underline{\lambda}(C))} \binom{X^C_2}{n_2(\underline{\lambda}(C))} \cdots \binom{X^C_m}{n_m(\underline{\lambda}(C))}
\]
where \(n_r(\mu)\) is the number of \(r\)'s in \(\mu\). Each such factorization of \(w\) corresponds to a coset \(C \in W_n / (W_m \times W_{n-m})\) that is stabilized by \(w\). For any representative \(s \in C\), \(p_m(s^{-1} w s)\) has labeled cycle type \(\underline \lambda\). So we conclude that
\[
\chi_{\Ind^{\FI_G}(V)_n} = \sum_{|\underline{\lambda}| = m} \chi_V(\underline{\lambda}) \prod_C \binom{X^C_1}{n_1(\underline{\lambda}(C))} \binom{X^C_2}{n_2(\underline{\lambda}(C))} \cdots \binom{X^C_m}{n_m(\underline{\lambda}(C))}
\]
where the left-hand-side is manifestly a fixed character polynomial independent of \(n\). Notice that this character polynomial has degree \(m\). The general result therefore again follows from (\ref{fin_res}) and Ramos \cite[Thm C]{Ra1} by semisimplicity.
\end{proof}

In particular, by taking \(g = e\) in \thref{charpoly}, we see that \(\dim V_n\) is given by a single polynomial for \(n \ge r + \min(m,r)\), which is \cite[Thm D]{Ra1}.

\subsection{Virtually polycyclic \(G\) and \(K_0\)-stability}
The above results only apply when \(G\) is finite, so that the group ring \(k[G]\) is semisimple. We would like to have a good analogue of representation stability for infinite virtually polycyclic \(G\). The answer is to pass to the Grothendieck group, where we essentially impose semisimplicity by fiat.

Recall that, for a ring \(R\), the Grothendieck group \(G_0(R)\) is defined to be the free abelian group generated by isomorphism classes \([M]\) of finitely-generated \(R\)-modules \(M\), modulo the relation that for any short exact sequence
\[ 0 \to M \to N \to M' \to 0,\;\text{ we have }\;[N] = [M] + [M']\]
The Grothendieck group \(K_0(R)\) is defined similarly, except that we only take finitely-generated \emph{projective} \(R\)-modules. The map \(K_0(R) \to G_0(R)\), \([P] \to [P]\) is well-defined and is called the \emph{Cartan map}.

Here we will be taking \(R = k[W_n]\), the group ring of \(W_n = G \wr S_n\) with \(G\) a virtually polycyclic group. Since \(W_n\) is itself virtually polcyclic, it is then known that the Cartan map is an isomorphism, so the two notions of Grothendieck group coincide. We will use the notation \(K_0(W_n) := K_0(k[W_n]) \cong G_0(k[W_n])\), as much to avoid clashing with the group \(G\), but we abuse notation by writing \([M] \in K_0(W_n)\) for any finitely-generated \(W_n\)-module \(M\), which makes sense since the Cartan map is an isomorphism.

Grothendieck groups of virtually polycyclic groups have been much studied and are well understood; for a good overview, see \cite[Ch. 8]{Pa}. For one thing, such groups are always torsion-free and finitely generated, and thus isomorphic to \(\mathbb{Z}^r\). One powerful result is \emph{Moody's induction theorem} \cite{Mo}, which says that for \(G\) virtually polycyclic, \(K_0(k[G])\) is generated by induced modules of the finite subgroups of \(G\). This means that, at the level of \(K_0\), the representation theory of \(G \wr S_n\) for \(G\) virtually polycyclic works in the same way as for \(G\) finite that was described in \S2.1, as follows. 

The monoid \(K^+_0(G)\) of \emph{positive} classes---that is, those actually represented by honest, and not just virtual, representations---is isomorphic to \(\mathbb{N}^r\), and thus has a unique basis \(X = \{\chi_1, \dots, \chi_r\}\). Pick representative \(G\)-modules \(M_1, \dots, M_r\), so that \([M_i] = \chi_i\). If \(\underline \lambda\) is a partition-valued function on \(X\), let \(L(\underline \lambda)\) be the associated representation of \(W_n\) defined in \S2.1, using the \(M_i\) as the representations associated to \(\chi_i\). Of course, this is only defined up to a choice of \(\{M_i\}\), but it does give a well-defined element \([L(\underline\lambda)] \in K_0(G)\). Moody's induction theorem then has the following consequence:

\begin{prop}
The collection \(\{[L(\underline \lambda)] : \|\underline\lambda\| = n\}\) forms a basis for \(K_0(W_n)\).
\end{prop}
\begin{proof}
The fact that the \([L(\underline\lambda)\)]'s span \(K_0(W_n)\) is the actual consequence of Moody's induction theorem. For linear independence---that is, injectivity of the map \(\mathbb{Z}\langle [L(\underline \lambda)] \rangle_{\underline \lambda} \to K_0(W_n)\)---see e.g. \cite[Exer 35.3]{Pa}.
\end{proof}

The reason \(K_0\) is relevant to us is because of the following generalization of the resolution (\ref{fin_res}) to virtually polycyclic \(G\), due to Nagpal-Snowden:
\begin{thm}[Nagpal-Snowden \cite{NS}]
If \(G\) is virtually polycyclic, and \(V\) is a finitely-generated \(\FI_G\)-module, then for \(n\) sufficiently large, there is a filtration
\[0 = V^{(0)}_n \subset V^{(1)}_n \subset \cdots \subset V^{(m)}_n = V_n\]
such that the successive quotients \(V_n^{(k+1)}/V_n^{(k)}\) are of the form \(\bigoplus_i \Ind^{\FI_G}(W_i)_n\).
\end{thm}
We are now ready to formulate the relevant generalization of representation stability. For \(G\) a virtually polycyclic group, say that an \(\FI_G\)-module \(V\) satisfies \emph{\(K_0\)-stability} if for all sufficiently large \(n\), there is a decomposition
\[ [V_n] = \sum_{\underline\lambda} c(\underline\lambda) [L(\underline\lambda[n])]\]
where \(c(\underline\lambda)\) does not depend on \(n\). We then obtain the following.

\begin{thm}[\bfseries \(K_0\)-stability]
If \(G\) is virtually polycyclic and \(V\) is a finitely-generated \(\FI_G\)-module, then \(V\) satisfies \(K_0\)-stability.
\end{thm}
\begin{proof}
By Theorem 2.4, for \(n\) sufficiently large, we can write
\[ [V_n] = \sum_i [\Ind^{\FI_G}(V_i)_n] - \sum_j [\Ind^{\FI_G}(W_j)_n] \]
for some collection of representations \(\{V_i\}\), \(\{W_j\}\), independent of \(n\). It therefore suffices to verify that the individual terms \(\Ind^{\FI_G}L(\underline\lambda)\) satisfy \(K_0\)-stability.

But now the proof given in \cite[Thm 1.10]{GL3} of representation stability of \(\Ind^{\FI_G} L(\underline\lambda)\) applies verbatim, since it does not use anything about the group \(G\); it just uses Pieri's formula and other facts about the representation theory of \(S_n\).
\end{proof}

\subsection{Arbitrary \(G\) and finite presentation degree}
Until recently, the Noetherian property was seen as the lynchpin of the theory of \(\FI\)-modules and related categories. However, coming out of the work of Church-Ellenberg \cite{CE} on homological properties of \(\FI\)-modules, a new perspective has emerged that shifts the emphasis from finite generation of modules to finite \emph{presentation degree} of modules, e.g. in the work of Ramos \cite{Ra2} and Li \cite{Li}. In one sense, this perspective is more ``constructive'', because possessing the knowledge of both the degree of generation \emph{and} the degree of relation of a module gives us quantitative control over various stability properties of the module, as we have seen. The Noetherian property then tells us that \emph{any} finitely generated module is necessarily finitely presented, which is an important fact but no longer at the absolute center of the theory. Appealing to Noetherianity is also necessarily ineffective, since we are no longer able to say what the relation degree is, and thus lose effective bounds on stability.

At the same time, this shift in perspective allows us to expand our scope to situations where there is no hope of finite generation. For example, we will see later on examples of \(\FI_G\)-modules \(V\) that are not finitely generated simply because the individual pieces \(V_n\) are not finitely-generated \(W_n\)-representations. Nevertheless, we will be able to prove that \(V\) is still generated in finite degree, in the sense that all the generators of \(V\) (an infinite number) live only in \(V_1, \dots V_m\) for some \(m\), \emph{and} that \(V\) is related in finite degree. Of course, in a sense such \(V\) are much \emph{less} constructive then anything in the finitely-generated world.

This shift also allows us to leave behind the requirement that \(G\) be virtually polycyclic. Indeed, as we have seen, this requirement is based in the fact that, for any kind of Noetherianity to get off the ground, we certainly need the ring \(k[G]\) to be Noetherian, which as far as we know is only true when \(G\) is virtually polycyclic. However, once we have accepted that modules can be infinitely generated, and only care about the \emph{degree} that that they are generated (and related) in, we can leave this need behind.

The central result that informs this perspective is the following, proved simultaneously by Ramos and Li:
\begin{thm}[{\cite[Thm B]{Ra2}, \cite[Prop 3.4]{Li}}] \thlabel{coh}
For any group \(G\), the category of \(\FI_G\)-modules presented in finite degree is abelian.
\end{thm}

This is the analogue of Noetherianity for finite presentation degree, since it allows us to argue that kernels of maps between \(\FI_G\)-modules presented in finite degree are still presented in finite degree, and therefore for example to chase being presented in finite degree through a spectral sequence. Indeed, from a certain point of view \thref{coh} is the fundamental fact, and Noetherianity as we have seen it is so far is just a consequence of \thref{coh} and the fact that the individual group rings \(k[G]\) are Noetherian.

What we lose at this level of generality is any ability to refer back to stability results in terms of things like ``representation stability'' or ``stability of character polynomials'' that are not couched explicitly in terms of \(\FI_G\)-modules. All we can say is that the \(\FI_G\)-modules in question are presented in finite degree, which perhaps is less interesting to someone who only cares about the individual \(W_n\)-representations \(V_n\). At the same time, since these are infinitely-generated representations of infinite groups, it is hard to say much about the individual representations.

\subsection{\(\FI_G\sharp\)-modules}
The classification of projective modules provided above means that even when \(G\) is infinite, if an \(\FI_G\) module \(V\) is projective, then there is still a compact description of the representation theory of \(V\). We would therefore like to be able to determine when an \(\FI_G\)-module is projective, so that it has such a description. \cite{CEF} provide just such a method: they define a category \(\FI\sharp\) with an embedding \(\FI \hookrightarrow \FI\sharp\) so that \(\FI\sharp\)-Mod is precisely the category of projective \(\FI\)-modules, so a module is projective just when it extends to an \(\FI\sharp\)-module. Their construction and proof of equivalence carry over to the setting of \(\FI_G\), as Wilson \cite{Wi} proved for the case \(G = \mathbb{Z}/2\).

So define \(\FI_G\sharp\) to be the category of \emph{partial morphisms} of \(\FI_G\): the objects are still finite sets, but a map \(X \to Y\) is given by a pair \((Z, f)\), where \(Z \subset X\) and \(f: Z \to Y\) is a \(G\)-map. Composition of morphisms is defined by pullback, i.e. with the domain the largest set on which the composition is defined. Then there is a natural structure of \(\Ind^{\FI_G}(V)\) as an \(\FI_G\sharp\)-module. Furthermore, we obtain the following.
\begin{prop} \thlabel{sharp_decomp}
Any \(\FI_G\sharp\) module is isomorphic to \(\,\bigoplus_i \Ind^{\FI_G}(W_i)\) for some representations \(W_i\) of \(G_i\).
\end{prop}
\begin{proof}
This was proved for \(G\) trivial in \cite[Thm 4.1.5]{CEF}, and for \(G = \mathbb{Z}/2\) in \cite[Thm 4.42]{Wi}. As Wilson explains, the proof in \cite{CEF} applies almost verbatim: the only change that needs to be made is to the definition of the endomorphism \(E: V \to V\), which should be defined as follows, for \(m \ge n\),
\begin{align*}
E_m&: V_m \to V_m \\
E_m& = \sum_{\substack{S \subset [m] \\ |S| = n}} I_S, \;\; \text{where } I_S = (S, \iota) \in \Hom_{\FI_G\sharp}([m], [m]) \text{ with } \iota: S \hookrightarrow [m] \text{ the inclusion}.
\end{align*}
\end{proof}

\begin{cor}
\thlabel{sharp_range}
If \(V\) is an \(\FI_G\sharp\)-module generated in degree \(m\), then \(\chi_V\) is given by a single character polynomial of degree \(\le m\), and satisfies representation stability with stability degree \(\le 2m\).
\end{cor}
\begin{proof}
This follows from \cite[Thm 1.10]{GL3} and \thref{charpoly}.
\end{proof}

\subsection{Tensor products and \(\FI_G\)-algebras}
Here we proceed to generalize the notions introduced in \cite[\S4.2]{CEF} from \(\FI\) to \(\FI_G\). 

Given \(\FI_G\)-modules \(V\) and \(V'\), their tensor product \(V \otimes V'\) is the \(\FI_G\)-module with \((V \otimes V')_n = V_n \otimes V'_n\), where \(\FI_G\) acts diagonally. 

A \emph{graded \(\FI_G\)-module} is a functor from \(\FI_G\) to graded modules, so that each piece is graded, and the induced maps respect the grading. If \(V\) is graded, each graded piece \(V^i\) is thus an \(\FI_G\)-module. If \(V\) and \(W\) are graded, the tensor product \(V \otimes W\) is graded in the usual way. Say that \(V\) is of \emph{finitely-generated type} if each \(V^i\) is finitely generated. Say that \(V\) is of \emph{finite type} if it is of finitely-generated type and furthermore each \(V^i_n\) is finite-dimensional. Notice if \(G\) is a finite group, these two notions coincide.

Similarly, an \(\FI_G\)-algebra is a functor from \(\FI_G\) to \(k\)-algebras, which can also be graded. Usually our algebras will be associative, but in order to deal with Lie algebras without having to invoke anything as fancy as Poincar\'e-Birkhoff-Witt, we do not assume associativity here. We can also define graded co-\(\FI_G\)-modules and algebras, as functors from \(\FI_G^{\text{op}}\). A (co-)\(\FI_G\)-algebra \(A\) is \emph{generated (as an \(\FI_G\)-algebra)} by a submodule \(V\) when each \(A_n\) is generated as an algebra by \(V_n\). 

Finally, there is another type of tensor product that we will need. Suppose \(V\) is graded \(G\)-module with \(V^0 = k\). Then the space \(V^{\otimes \bullet}\) defined by \((V^{\otimes \bullet})_n = V^{\otimes n}\) has the structure of an \(\FI_G\sharp\)-module, as in \cite[Defn 4.2.5]{CEF}, with the morphisms permuting and acting on the tensor factors.

The following theorem characterizes the above constructions.
\begin{thm} \thlabel{alg} Let \(G\) be any group.\begin{enumerate}
\item Let \(V\) and \(V'\) be \(\FI_G\)-modules generated in degree \(m\) and \(m'\). Then \(V \otimes V'\) is generated in degree \(m + m'\). If \(V\) is finitely generated and \(V'\) is finite type, then \(V \otimes V'\) is finitely generated.

\item Let \(A\) be a graded (co-)\(\FI_G\)-algebra generated by a graded submodule \(V\), where \(V^0 = 0\) and \(V\) is generated as an \(\FI_G\)-module in degree \(m\). Then the \(i\)-th grades piece \(A^i\) is generated in degree \(m \cdot i\). If \(V\) is of finite type, then \(A\) is of finite type.

\item If \(V\) is a graded \(G\)-module with \(V^0 = k\), then \(V^{\otimes \bullet}\) is an \(\FI_G\sharp\)-module whose \(i\)-th graded piece is generated in degree \(i\). If \(V\) is of finite type, then \(V^{\otimes \bullet}\) is of finite type. 

\item If \(X\) is a connected \(G\)-space, then \(H^*(X^{\bullet}; k)\) is an \(\FI_G\sharp\)-algebra whose \(i\)-th graded piece is generated in degree \(i\). If \(H^*(X; k)\) is of finite type, then \(H^*(X^{\bullet}; k)\) is of finite type.
\end{enumerate}
\end{thm}
\begin{proof} \ \begin{enumerate}
\item This was proved by Sam-Snowden \cite[Prop 3.1.6]{SS2} for \(G\) finite, and their proof applies verbatim even when \(G\) is infinite to show that \(V \otimes V'\) is always generated in degree \(m + m'\), though not always finitely.

For the second part, for each \(k \le m + m'\), we know \((V \otimes V')_k = V_k \otimes V'_k\) is a finitely-generated \(W_k\)-module, since the tensor product of a finitely-generated \(W_k\)-module and a finite-dimensional module is finitely generated. The result follows.

\item This was proved for \(G\) trivial in \cite[Thm 4.2.3]{CEF} by appealing to the free nonassociative algebra \(k\{V\}\) generated by a vector space, and their proof applies verbatim even when \(G\) is infinite to show that \(A^i\) is always generated in degree \(m \cdot i\), though not always finitely. The second part follows from the second part of (a).

\item This was proved for \(G\) trivial in \cite[Prop 4.2.7]{CEF}, but here we have simplified the assumptions, by having \(V\) just be a single \(G\)-module rather than a whole \(\FI_G\)-module, and so the proof is simpler. To wit, the \(i\)-th graded piece is
\[ (V^{\otimes n})^i = \bigoplus_{k_1 + \cdots + k_n = i} V^{k_1} \otimes \cdots \otimes V^{k_n} \]
At most \(i\) of the nonnegative integers \(k_1, \dots, k_n\) can be nonzero. Let \(k_{j_1}, \dots, k_{j_l}\) be the subsequence of nonzero integers, so \(l \le i\). Then for any pure tensor
\[ v = v_1 \otimes \cdots \otimes v_n \in V^{k_1} \otimes \cdots \otimes V^{k_n},\]
since \(V^0 = k\), we can take
\[ w = \left(\prod_{l \notin \{k_j\}} v_l\right) v_{k_{j_1}} \otimes \cdots v_{k_{j_l}} \]
and then the inclusion \(j: [l] \hookrightarrow [n]\) clearly induces \(j_* w = v\). So by linearity, \((V^{\otimes n})^i\) is generated in degree \(i\).

If each \(V^i\) is finite-dimensional, then each \(V^{k_1} \otimes \cdots \otimes V^{k_l}\) is finite-dimensional for \(l = 1, \dots, i\). Therefore \((V^{\otimes n})^i\) is an \(\FI_G\)-module of finite type.

\item This was proved for \(G\) trivial in \cite[Prop 6.1.2]{CEF}, and their proof applies verbatim. It essentially follows from (3) and the K\"unneth formula. As \cite{CEF} explain, technically sometimes a sign is introduced in permuting the order of tensor factors, but this does not change the proof of (3). The degree 0 part is \(k\) by connectivity.
\end{enumerate}
\end{proof}

\section{Orbit configuration spaces}
Let \(M\) be a manifold with a free and properly discontinuous action of a group \(G\), so that \(M \to M/G\) is a cover. Define the \emph{orbit configuration space} by:
\[\Conf^G_n(M) = \{(m_i) \in M^n \mid G m_i \cap G m_j = \emptyset \text{ for } i \ne j\}\]
This was first considered by Xicot\'encatl in \cite{Xi} and later investigated in e.g. \cite{Co}, \cite{FZ}, \cite{CX}, \cite{CCX}. There is a covering map \(\Conf^G_n(M) \to \Conf_n(M/G)\) with deck group \(G^n\), given by \((m_i) \mapsto (G m_i)\). Thus another way to think of \(\Conf^G_n(M)\) is as the space of configurations in \(M\) that do not degenerate upon projection to \(M/G\), or as configurations in \(M/G\) which also keep track of a lift in \(M\) of each point in the configuration.

If \(N\) is a normal subgroup of \(G\), there is an intermediate cover
\[\Conf^G_n(M) \to \Conf^{G/N}_n(M/N) \to \Conf_n(M/G) \]
where the first map has deck group \(N^n\), and the second has deck group \((G/N)^n\). Also, notice that if \(G\) is finite, there is an embedding
\[ \Conf^G_n(M) \hookrightarrow \Conf_{|G|n}(M), \;\; (m_1, \dots m_n) \mapsto (g_1 m_1, g_2 m_1, \dots, g_{|G|} m_n) \]

\subsection{\(\FI_G\)-structure and finite generation for finite groups}
For any \(G\) acting discretely and properly discontinuously on \(M\), if we write \(\left(\Conf^G(M)\right)_n = (\Conf^G_n(M))\), then \(\Conf^G(M)\) is a co-\(\FI_G\)-space: given \(a: [m] \hookrightarrow [n]\) and \((g_i) \in G^m\), there is a map
\begin{align*}
(a, g)^*: \Conf^G_n(M) &\to \Conf^G_m(M) \\
(m_i) &\mapsto \left(g_i m_{a(i)}\right)
\end{align*}
In particular, if \(G\) is trivial we recover the usual ordered configuration space, and if \(M = \mathbb{C}^*\) and \(G = \mathbb{Z}/2\mathbb{Z}\) acting as multiplication by \(-1\), we obtain the type BC hyperplane complement from \cite{Wi}. Composing with the contravariant cohomology functor, we see that \(H^*(\Conf^G(M), k)\) has the structure of an \(\FI_G\)-module over the field \(k\).

We are therefore interested in orbit configuration spaces for which \(G\) is virtually polycyclic, so that we can apply the results on \(\FI_G\)-modules from \S2. Interesting examples include:
\begin{itemize}
\item \(G = \mathbb{Z}/2\) acting antipodally on \(M = S^m\), so that \(M/G = \mathbb{R}P^m\). This was analyzed by Feichtner and Ziegler in \cite{FZ}. Their computation of the cohomology in Thm 17 shows that \(\dim H^i(\Conf^G_n(S^m); \mathbb{Q})\) is bounded by a polynomial in \(n\). We strengthen this by proving that it is in fact equal to a polynomial for \(n \gg 0\). Furthermore, the \(G^n\) action on \(H^i(\Conf^G_n(S^m))\) was analyzed in \cite{GGSX}: in particular, their Prop 6.6 is a sort of weak form of representation stability.
\item \(G = \mathbb{Z}/2\) acting by a hyperelliptic involution on a 4-punctured \(\Sigma_g\), with quotient a 4-punctured sphere
\item Any finite cover \(\Sigma_h \to \Sigma_g\) (so that \(h = |G|(g - 1) + 1\)), with \(G\) the deck group of the covering.
\item \(M = \mathbb{R}^2\), with \(G\) a lattice isomorphic to \(\mathbb{Z}^2\), and \(M/G\) a torus
\item \(M = \Conf_d(\mathbb{C})\) for some fixed \(d\), with \(G = S_d\), so that we are looking at the iterated configuration space \(\Conf^{S_d}_n(\Conf_d(\mathbb{C}))\) and its quotient \(\UConf_n(\UConf_d(\mathbb{C}))\). 
\item \(M = S^3\), with \(G\) any finite subgroup of \(\SO(4)\), and \(M/G\) a spherical 3-manifold
\item Baues \cite{Bau} proved that every torsion-free virtually polycyclic group \(G\) acts discretely, properly discontinuously, and cocompactly on \(\mathbb{R}^d\) for some \(d\), and furthermore, the quotient spaces \(\mathbb{R}^d/G\) precisely comprise the \emph{infra-solvmanifolds}. Thus for any such \(G\), consider \(\Conf^G_n(\mathbb{R}^d)\).
\end{itemize}

A straightforward transversality argument (e.g., \cite[Thm 1]{Bi}) shows that if \(\dim M \ge 3\), the map \(\Conf_n(M) \hookrightarrow M^n\) induces an isomorphism on \(\pi_1\). Therefore \(\Conf^{\pi_1(M)}_n(\widetilde{M})\) is in fact the universal cover of \(\Conf_n(M)\) in dimension \(\ge 3\), which provides further motivation as to why orbit configuration spaces are natural to study. Note that this does not make the last example trivial (i.e., contractible), since if \(\dim M > 2\), \(\Conf_n(M)\) need not be aspherical: its homotopy groups only agree with \(M^n\) up to \(\dim(M)-2\).

We first consider the case where \(G\) is finite.

\begin{thm}[\bfseries Cohomology of orbit configuration spaces] \thlabel{confg_fg} Let \(k\) be a field, let \(M\) be a connected, orientable manifold of dimension at least 2 with each \(H^i(M; k)\) finite-dimensional, and let \(G\) be a finite group acting freely on \(M\). Then the \(\FI_G\)-algebra \(H^*(\Conf^G(M); k)\) is of finite type. \end{thm}

Following Church-Ellenberg-Farb-Nagpal \cite{CEFN}, we could adapt \thref{confg_fg} to handle \(\mathbb{Z}\) coefficients. As \cite{CEFN} mention, the proof with \(\mathbb{Z}\) coefficients is essentially identical to the one with field coefficients: the difference is that one of the inputs, the analogue of our \thref{alg}.4, becomes harder to prove. However, their proof of this analogue, \cite[Lemm. 4.1]{CEFN}, is readily adaptable to our context. We do not make use of this except in \S3.4.

\begin{proof}
The proof is based on a modification of Totaro's \cite[Thm 1]{To}. Following Totaro, consider the Leray spectral sequence associated to the inclusion \(\iota: \Conf^G_n(M) \hookrightarrow M^n\). This spectral sequence has the form
\begin{equation}\label{tot_ss}
H^i(M^n; R^j \iota_* k) \implies H^{i+j}(\Conf^G_n(M); k)
\end{equation}
where \(R^j \iota_* k\) is the sheaf on \(M^n\) associated to the presheaf
\[U \mapsto H^j(U \cap \Conf^G_n(M); k)\]

As in Totaro's proof, the sheaf \(R^j \iota_* k\) vanishes outside the appropriate ``fat diagonal'', which in this case is the union of the subspaces \(\Delta_{a,g,b} = \{(m_i) \in M^n \mid m_a = g \cdot m_b\}\), for \(1 \le a < b \le n\) and \(g \in G\). Consider a point in the fat diagonal,
\[x = (x_1, g_{11} x_1, \dots g_{1i_1} x_1, \dots x_s, g_{s1} x_s, \dots g_{s i_s} x_s).\]
 Since \(G\) acts properly discontinuously, take a neighborhood of each \(x_j\) small enough to be disjoint from all translates of all the other neighborhoods. Then use a Riemannian metric to identify each of these with the tangent space \(T_{x_j} X\), \(T_{g_{j 1} x_j} X\), etc. \(dg_{j k}(x_j)\) is then an isomorphism from \(T_{x_j} X\) to \(T_{g_{jk} x_j} X\), and the condition that a point \(m_1\) near \(x\) and \(m_2\) near \(g x\) satisfy \(m_2 = g m_1\) becomes, upon passing to the tangent space and under the isomorphism \(dg\), simply the condition that \((v,w) \in (T_x X)^2\) satisfy \(v = w\). Thus, for a neighborhood \(U\) of \(x\) small enough so that the inverse exponential map is a diffeomorphism,
\[
(R^j \iota_* k)_x = H^j(U \cap \Conf^G_n(M); k) = H^j(\Conf_{i_1}(T_{x_1} X) \times \cdots \times \Conf_{i_s}(T_{x_s} X))
\]
Thus, the local picture looks exactly the same as in \(\Conf_n\), which it should since there is a covering map \(\Conf_n^G(M) \to \Conf_n(M/G)\). So as in Totaro, we get generators of \(\Conf_n(\mathbb{R}^d)\), where \(\dim M = d\). However, for \(\Conf_n(M)\), we got just one copy of each generator \(e_{ab}\), coming from the diagonal \(\{m_a = m_b\}\). Here, however, we get a generator \(e_{a,g,b}\) for each \(g \in G\), coming from each \(\Delta_{a,g,b}\). The permutation action of \(W_n\) on the \(\{\Delta_{a,g,b}\}\) induces an action on the \(\{e_{a,g,b}\}\), which is given by
\begin{equation} \label{action}
(\sigma, \vec h) \cdot e_{a,g,b} = e_{\sigma(a), h_a g h_b^{-1}, \sigma(b)}
\end{equation}
As in Totaro, we can explicitly write down the relations that these \(e_{a,g,b}\) satisfy:
\begin{gather} \begin{split} \label{presentation}
e_{a,g,b} &= (-1)^d e_{b,g^{-1},a} \\
e_{a,g,b}^2 &= 0 \\
e_{a,g,b} \wedge e_{b,h,c} &= (e_{a,g,b} - e_{b,h,c}) \wedge e_{a,gh,c}
\end{split} \end{gather}
To conclude, we use the argument from \cite[Thm 6.2.1]{CEF}. To wit, because the Leray spectral sequence is functorial, all of the spectral sequences of \(\Conf_n^G(M) \hookrightarrow M^n\), for each \(n\), collected together form a spectral sequence of \(\FI_G\)-modules. As we just described, the \(E_2\) page is generated by \(H^*(M^n; k)\) and
the \(\FI_G\)-module spanned by the \(e_{a,g,b}\). This latter is evidently finitely-generated, since it is just generated in degree 2 by \(e_{1,e,2}\), and the former is of finite type by \thref{alg}.4, so therefore the \(E_2\) page as a whole is of finite type. The \(E_\infty\) page is a subquotient of \(E_2\), so by Noetherianity it is of finite type, and therefore \(H^*(\Conf^G(M); k)\) is of finite type.
\end{proof}
We pause briefly to dwell on the permutation action (\ref{action}) of \(W_n\) on the module \(V_n\) spanned by \(\{e_{a,g,b}\}\), since it will come up repeatedly. Let \(k[G]\) be the representation of \(G \times G\) where the first factor of \(G\) acts by multiplication on the left, and the second by multiplication on the right by the inverse (two commuting left actions). Therefore
\[V_n = \Ind_{W_2 \times W_{n-2}}^{W_n} k[G] \otimes k = \Ind^{\FI_G}(k[G])_n \]

Recall that, as a \((k[G], k[G]^{\text{op}})\)-bimodule, the regular representation has the following decomposition into irreducibles:
\[ k[G] = \bigoplus_{\chi \in \Irr(G)} V_{\chi} \boxtimes (V_{\chi})^* \]
Thus if we turn this into a \(k[G] \otimes k[G]\)-module by having the right factor act by \(g^{-1}\), this becomes
\[k[G] = \bigoplus_{\chi \in \Irr(G)} V_{\chi} \boxtimes V_{\chi} = \bigoplus_{\chi \in \Irr(G)} L( (2)_\chi) \]
as \(k[G] \otimes k[G] = k[G \times G]\)-modules. Therefore
\begin{equation} \label{irr_decom}
V = \bigoplus_{\chi \in \Irr(G)} \Ind^{\FI_G}((2)_{\chi}).
\end{equation}
as an \(\FI_G\)-module, so it is manifestly an \(\FI_G\sharp\)-module. In particular, if \(G\) is trivial, we obtain \(\Ind^{\FI}((2)) = \Sym^2 k^n / k^n\), consistent with the computation done in \cite{CEF}.

\begin{cor}
Let \(M\) be a connected, orientable manifold of dimension at least 2 with each \(H^i(M; \mathbb{Q})\) finite-dimensional, let \(G\) be a finite group acting freely on \(M\), and let \(k\) be a splitting field for \(G\) of characteristic 0. Then for each \(i\), the characters of the \(W_n\)-representations \(H^i(\Conf^G_n(M); k)\) are given by a single character polynomial for all \(n \gg 0\).
\end{cor}

\thref{confg_fg} has another consequence, which as far as we can tell is a new result (recall that \cite[Thm 6.2.1]{CEF} only applied to orientable manifolds).

\begin{cor}
Let \(M\) be a connected, non-orientable manifold of dimension at least \(2\) with \(H^*(M; \mathbb{Q})\) of finite type. Then the \(\FI\)-algebra \(H^*(\Conf(M); \mathbb{Q})\) is of finite type.
\end{cor}
\begin{proof}
Consider the orientation cover \(\widetilde{M} \to M\), which has deck group \(G = \mathbb{Z}/2\). Since there is a covering map \(\Conf^G_n(\widetilde{M}) \to \Conf_n(M)\), with deck group \(G^n\), then by transfer there is an isomorphism
\[H^*(\Conf_n(M); \mathbb{Q}) \cong \left(H^*(\Conf^G_n(\widetilde{M}); \mathbb{Q})\right)^{S_n}\]
By \thref{confg_fg}, \(H^*(\Conf^G_n(\widetilde{M}); \mathbb{Q})\) is a finite type \(\FI_G\)-algebra, and therefore \(H^*(\Conf_n(M); \mathbb{Q})\) is a finite type FI-algebra.
\end{proof}

Homological stability for non-orientable manifolds, which is a consequence of this corollary, was proven by Randal-Williams in \cite{RW}.

\subsection{Dealing with infinite groups}
We pause to explain the complications that arise when \(G\) is infinite, before proceeding to at least partially resolve them. As a toy example, forgetting for a moment the setting of \(\FI_G\) and sequences of spaces, consider the space \(X = S^1 \vee S^1 = K(F_2, 1)\), with one loop called \(a\) and the other \(b\). Let \(Y\) be the \(G := \mathbb{Z}\) cover associated to the kernel of the map \(F_2 \to \mathbb{Z}\), \(a \mapsto 0, b \mapsto 1\). Hence \(Y\) is an infinite sequence of line segments labeled \(b\) joining loops labeled \(a\). So \(Y\) is homotopy equivalent to a wedge of infinitely many circles, and thus \(H_1(Y; k)\) has infinite rank. However, notice that the covering group \(G\) acts on \(Y\), and that \(H_1(Y; k) \cong k[G] = k[b^{\pm 1}]\) as \(G\)-modules.

In particular, \(H_1(Y; k)\) is finitely-generated as a \(G\)-module. However, looking at cohomology, \(H^1(Y; k) = \Hom_k(H_1(Y), k) = \Hom_k(k[G], k) = k^G\). Thus, \(H_1(Y; k)\) consists of finite linear combinations of elements of \(G\), and \(H^1(Y; k)\), infinite linear combinations. In particular, \(H^1(Y; k)\) is no longer finitely generated as a \(G\)-module. However, notice that it contains a dense submodule isomorphic to \(H_1(Y; k)\).

Now we see why the proof of \thref{confg_fg} does not suffice when \(G\) is infinite: in general, the \(E_2\) page is not just be generated by the \(e_{a,g,b}\), which is to say, by finite linear combinations of them, but instead it is generated by all infinite linear combinations of them. Therefore the \(E_2\) page is in general not a finitely-generated \(\FI_G\)-module. 

If we are willing to settle for ``presented in finite degree''---for example, if \(G\) is not virtually polycyclic---then this is good enough:
\begin{thm} \thlabel{confg_coh}
Let \(M\) be a connected, orientable manifold of dimension at least 2. Then the \(\FI_G\)-algebra \(H^*(\Conf^G_n(M); k)\) is presented in finite degree.
\end{thm}
\begin{proof}
The argument from \thref{confg_fg} carries over essentially directly. The \(E_2\) page of the spectral sequence (\ref{tot_ss}) is generated as a \(k\)-algebra by \(H^*(M^n; k)\) and the dual space to \(\langle e_{a,g,b} \rangle\), that is, the space of infinite linear combinations
\[ \sum_{a,g,b} v_{a,g,b} e_{a,g,b} = \sum_{1 \le a < b \le n} \left( \sum_{g \in G} v_{a,g,b} e_{a,g,b} \right) \]
This space is evidently generated as an \(\FI\)-module by those sums of the form \(\sum_{g \in G} v_{g} e_{1,g,2}\), since we can get all other \(a,b\) by appropriate permutations. These generators evidently live in degree 2, and the relations (\ref{presentation}) all live in degree 3, so that the \(E_2\) page is presented in finite degree. By \thref{coh}, taking successive pages in the spectral sequences preserves being presented in finite degree, as does passing from the \(E_\infty\) page to the final cohomology. So we conclude that \(H^*(\Conf^G_n(M); k)\) is presented in finite degree.
\end{proof}

However, if we want to preserve finite generation, the analysis at the beginning of this subsection suggests that the correct thing to look at is actually \(H_\bullet(\Conf^G(M); k)\) instead. Unfortunately, on the face of it, \(H_\bullet(\Conf^G(M); k)\) is a co-\(\FI_G\)-module, so since the maps go ``in the wrong way'', it is never finitely generated.

However, when the quotient \(M/G\) is an \emph{open} manifold, we obtain the following generalization of \cite[Prop 6.4.2]{CEF}.
\begin{thm}[\bfseries Orbit configuration spaces of open manifolds]
\thlabel{boundary_sharp}
Let \(N\) be the interior of a connected, compact manifold \(\overline{N}\) of dimension \(\ge 2\) with nonempty boundary \(\partial\overline{N}\), and let \(\pi: M \to N\) be a \(G\)-cover, so that \(G\) acts freely and properly discontinuously on \(M\). Then \(\Conf^G(M)\) has the structure of a homotopy \(\FI_G\sharp\)-space, that is, a functor from \(\FI_G\sharp\) to hTop, the category of spaces and homotopy classes of maps.
\end{thm}
\begin{proof}
We follow the argument in \cite{CEF}. Fix a collar neighborhood \(S\) of one component of \(\partial \overline{N}\), let \(R = \pi^{-1}(S)\) and let \(R_0\) be a connected component of \(R\), and fix a homeomorphism \(\Phi: M \cong M \setminus \overline{R}\) isotopic to the identity (\(\Phi\) and the isotopy both exist by lifting). For any inclusion of finite sets \(X \subset Y\), define a map
\[ \Psi^Y_X: \Conf^G_X(M) \to \Conf^G_Y(M)\]
up to homotopy, as follows. First, if \(Y = X\), set \(\Psi^Y_X = \id\). Next, note that a configuration in \(\Conf^G_X(M)\) is just an embedding \(X \hookrightarrow M\) (that stays injective upon composition with \(\pi\)). So fix an element \(q_X^Y: (Y - X) \hookrightarrow R_0\) of \(\Conf^G_{Y - X}(R_0)\). Then any embedding \(f: X \hookrightarrow M\) in \(\Conf^G_X(M)\) extends to an embedding \(\Psi^Y_X(f): Y \hookrightarrow M\) by
\[ \Psi^Y_X(f)(t) = \begin{cases} \Phi(f(t)) & t \in X \\ q^Y_X(t) & t \notin X \end{cases} \]
The image of \(\pi \circ \Phi\) is disjoint from \(S\), while the image of \(\pi \circ q^Y_X\) is contained in \(S\), so the above map never send points in \(X\) into the same \(G\)-orbit as it sends points outside of \(X\), and therefore this does give an element of \(\Conf^G_Y(M)\). Furthermore, since \(\Conf^G_{Y - X}(R_0)\) is connected (since \(R_0\) is, and \(\dim R_0 \ge 2\)), different choices of \(q^Y_X\) give homotopic maps, so \(\Phi^Y_X\) is well-defined up to homotopy.

Now, an \(\FI_G\sharp\) morphism \(Z \to Y\) consists of an injection \(X \hookrightarrow Z\), an injection \(X \hookrightarrow Y\), and a tuple \(g: X \to G\). Normally if we were extending from an \(\FI_G\) structure, we would think of \(X\) as being a subset of \(Z\), but since we are extending from a \emph{co}-\(\FI_G\) structure, it is more natural to think of \(X\) as a subset of \(Y\), with an explicit map \(a: X \to Z\). The induced map is then given by
\begin{align*}
\Conf^G_Z(M) & \to \Conf^G_X(M) \xrightarrow{\Psi^Y_X} \Conf^Y(M) \\
(m_i) & \mapsto  (g_i m_{a(i)})
\end{align*}
It is straightforward to verify that this is functorial up to homotopy, as \cite[Prop 6.4.2]{CEF} do for trivial \(G\).
\end{proof}
In particular, when the conditions of \thref{boundary_sharp} hold, then \(H_*(\Conf^G(M); k)\) is an \(\FI_G\sharp\)-module. We want to argue that, when \(G\) is virtually polycyclic, \(H_*(\Conf^G(M); k)\) is finitely-generated type. To do this, we need the following.
\begin{prop} \thlabel{arnold}
Let \(A(G,d)\) be the \(\FI_G\sharp\)-algebra where \(A(G)_n\) has generators \(\{e_{a,g,b} \mid 1 \le a \ne b \le n, g \in G\}\) of degree \(d-1\) and action given by (\ref{action}), modulo the following relations:
\begin{align*}
e_{a,g,b} &= (-1)^d e_{b,g^{-1},a} \\
e_{a,g,b}^2 &= 0 \\
e_{a,g,b} \wedge e_{b,h,c} &= (e_{a,g,b} - e_{b,h,c}) \wedge e_{a,gh,c}
\end{align*}
Then \(A(G,d)\) is an \(\FI_G\sharp\)-module of finitely-generated type.
\end{prop}
\begin{proof}
First, by construction \(A(G,d)\) is presented in finite degree, so it remains to show that each \(A(G,d)_n\) is of finitely-generated type. For convenience, put \(D = d-1\), so that \(A(G,d)_n\) is only nonzero is degree divisible by \(D\). It is straightforward to verify (e.g., see [ARNOLD]) that \(A(G,d)^{iD}_n\) is spanned as a vector space by all products of the form
\begin{equation} \label{arnold_span}
v = e_{a_1, g_1, b_1} \wedge e_{a_2, g_2, b_2} \wedge \cdots \wedge e_{a_i, g_i, b_i} \;\; \text{ where } a_s < b_s, \;\;\; b_1 < b_2 < \cdots < b_i, \;\;\; g_s \in G
\end{equation}
We will describe an explicit procedure which, given such a \(v\), finds an element \(\vec h \in G^n\) so that
\[
v' := \vec h \cdot v = e_{a_1, e, b_1} \wedge e_{a_2, e, b_2} \wedge \cdots \wedge e_{a_i, e, b_i}
\]
that is, so that each \(g'_i = e\) in (\ref{arnold_span}). We construct \(\vec h\) inductively. To begin, we put a copy of \(g_1\) in the \(b_1\)-th coordinate of \(\vec h\), which cancels out the \(g_1\) in (\ref{arnold_span}). We multiply \(v\) by this partial \(\vec h\), which may have the effect of modifying later \(g_i\)'s. We then proceed and put a copy of (the new, modified) \(g_2\) in the \(b_2\)-th coordinate of \(\vec h\), which cancels the \(g_2\) in (\ref{arnold_span}). We can do this with no trouble because we know \(b_1 < b_2\). Again, we use this to modify \(v\) and proceed to \(b_3\), etc. Eventually we have constructed our \(\vec h\) and modified \(v\) so that each \(g_s = e\).

We therefore conclude that \(A(G,d)^{iD}_n\) is finitely generated as a \(G^n\)-module, so \emph{a fortiori} as a \(W_n\)-module. Therefore \(A(G,d)\) is of finitely-generated type.
\end{proof}

We therefore obtain the following.
\begin{thm}[\bfseries Homology of orbit configuration spaces]
\thlabel{confg_fgsharp}
Let \(N\) be the interior of a connected, compact manifold of dimension \(\ge 2\) with nonempty boundary. Let \(M \to N\) be a \(G\)-cover, with \(G\) virtually polycyclic, such that \(H_*(M)\) is of finite type. Then \(H_*(\Conf^G(M); k)\) is a finitely-generated type \(\FI_G\sharp\)-module.
\end{thm}
\begin{proof}
We still need to appeal to cohomology, in order to make use of the cup product structure. So the proof follows that of \thref{confg_fg}, but only considering the sub-\(\FI_G\)-module of the \(E_2\) page that actually is generated by \(H^*(M^n)\) and the \(e_{a,g,b}\), and not the infinite linear combinations of them. Notice that this is isomorphic to the \(E^2\) page associated to the appropriate spectral sequence computing the homology of \(\Conf^G(M)\) (to be technical, this comes from \emph{cosheaf} homology). 

This submodule of the \(E_2\) page is precisely the algebra described in \thref{arnold}. Thus it is of finitely-generated type, and therefore the \(E^2\) page for homology is an \(\FI_G\)-module of finitely-generated type. The same final argument from \thref{confg_fg} (which uses Noetheriantiy, since \(G\) is virtually polycyclic) thus shows that \(H_*(\Conf^G(M); k)\) is of finitely-generated type.
\end{proof}

\begin{cor}
Let \(N\) be the interior of a connected, compact manifold of dimension \(\ge 2\) with nonempty boundary, let \(M \to N\) be a \(G\)-cover, with \(G\) virtually polycyclic, \(M\) connected and orientable and \(H^*(M; \mathbb{Q})\) of finite type. Then for each \(i\), the \(W_n\)-representations \(H_i(\Conf^G_n(M); \mathbb{Q})\) satisfy \(K_0\)-stability.
\end{cor}

\subsection{Homotopy groups of configuration spaces}
The main application of the theory of \(\FI_G\)-modules until this point has been in Kupers-Miller's \cite{KM} work on the homotopy groups of configuration spaces. They prove that, for \(M\) a simply-connected manifold of dimension at least 3, the dual homotopy groups \(\Hom(\pi_i(\Conf_n(M)), \mathbb{Z})\) form a finitely-generated \(\FI\)-module. In \cite[\S5.2]{KM}, they sketch an extension of this result to the non-simply connected case. Kupers-Miller are naturally led to consider \(\Conf_n^{\pi_1(M)}(\widetilde{M})\), since as we have said it is the universal cover of \(\Conf_n(M)\) once \(\dim M \ge 3\), and so has the same higher homotopy groups as \(\Conf_n(M)\).

Our results on orbit configuration spaces are able to confirm most of Kupers-Miller's sketch, while also clarifying some oversights. As stated, their Prop 5.8 is not correct, since as we have seen, if \(G\) is infinite, in general we cannot conclude that \emph{co}homology is finitely generated. Instead, the best we can do is our \thref{confg_fg} and \thref{confg_fgsharp}, where we either assume that \(G\) is finite or that \(M\) is an open manifold. We therefore obtain the following. Note that, following Kupers-Miller, we work here with \(\mathbb{Z}\) coefficients, since as we mentioned, we could rework \thref{confg_fg} to use \(\mathbb{Z}\) coefficients.

\begin{thm}[\bfseries Homotopy groups of configuration spaces, take 1]
Let \(M\) be a connected manifold of dimension \(\ge 3\) with finite fundamental group \(G\) such that \(H^*(M)\) is finite-dimensional. For \(i \ge 2\), the dual homotopy groups \(\Hom(\pi_i(\Conf_n(M)), \mathbb{Z})\) and \(\Ext^1_{\mathbb{Z}}(\pi_i(\Conf_n(M)), \mathbb{Z})\) are finitely generated \(\FI_G\)-modules.

In particular, if \(k\) is a splitting field for \(G\) of characteristic 0, then for each \(i\), the characters of the \(W_n\)-representations \(\Hom(\pi_1(\Conf_n(M)), k)\) are given by a single character polynomial for \(n \gg 0\), and \(\{\Hom(\pi_1(\Conf_n(M)), k)\}\) satisfies representation stability.
\end{thm}

Applying \thref{confg_coh}, we obtain the following.
\begin{thm}[\bfseries Homotopy groups of configuration spaces, take 2]
Let \(M\) be a connected, orientable manifold of dimension at least 2. Then the dual homotopy groups \(\Hom(\pi_i(\Conf_n(M), \mathbb{Z})\) and \(\Ext^1_{\mathbb{Z}}(\pi_i(\Conf_n(M)), \mathbb{Z})\) are \(\FI_G\)-modules presented in finite degree.
\end{thm}

Applying \thref{confg_fgsharp}, we also obtain the following.

\begin{thm}[\bfseries Homotopy groups of configuration spaces, take 3]
Let \(M\) be the interior of a connected, compact manifold with nonempty boundary of dimension \(\ge 3\) such that \(G = \pi_1(M)\) is virtually polycyclic. For \(i\) at least 2, \(\pi_i(\Conf_n(M))\) is a finitely-generated \(\FI_G\sharp\)-module. In particular, \(\pi_i(\Conf_n(M)) \otimes \mathbb{Q}\) satisfies \(K_0\)-stability.

\end{thm}

\subsection{Intermediate configuration spaces and \(\FI \times G\)}
Now suppose that \(G\) is abelian, and consider the following map:
\[\psi: G^n \rtimes S_n \to G \times S_n,\;\; \psi(\vec{g}, \sigma) = (g_1 \cdots g_n, \sigma) \]
Since \(G\) is abelian and \(\sigma\) leaves the product of all the \(g_i\)'s invariant, \(\psi\) is a homomorphism. Its kernel is the subgroup \(\ker( \vec{g} \mapsto g_1 \cdots g_n) \subset G^n\), which as a group is just isomorphic to \(G^{n-1}\). Since \(G\) is a \(\mathbb{Z}\)-module, we can identify \(\ker \psi\) as \(G \otimes_\mathbb{Z} V\), where \(V\) is the standard representation of \(S_n\) of rank \(n-1\). We obtain the following diagram of covering spaces refining \(\Conf^G_n(M) \twoheadrightarrow \UConf_n(M/G)\):
\begin{equation*}
\begin{diagram}
&&\Conf^G_n(M)&& \\
&&\dTo_{G^{n-1}}&&\\
&& \Conf'_n(M/G)&& \\
&\ldTo^{S_n} && \rdTo^{G} & \\
\UConf'_n(M/G) &&&& \Conf_n(M/G) \\
&\rdTo_{G} && \ldTo_{S_n} & \\
&&\UConf_n(M/G)&&
\end{diagram}
\end{equation*}
where \(\Conf'_n(M/G)\) and \(\UConf'_n(M/G)\) are the intermediate covers. These do not have particularly nice descriptions in general, but if \(M\) is a Lie group and \(G\) a discrete central subgroup, so that \(M/G\) is also a Lie group, then
\[ \Conf'_n(M/G) = \{(m_1 \dots m_n; m) \in (M/G)^n \times M\,|\,m_i \ne m_j\,\forall i\ne j;\, \pi(m) = m_1 \cdots m_n\} \]
and \(\UConf'_n(M/G)\) is its quotient by \(S_n\). So \(\Conf'\) parameterizes configurations of points in the base space with a chosen lift of the \emph{product} of the points. When \(M/G = \mathbb{C}^*\), these spaces are closely related to the Burau representation, as we shall see.

Define \(\FI \times G\) to be the usual product of categories, where we think of \(G\) as a category with one object. An \(\FI \times G\)-module is thus a sequence of \(G\)-representations equipped with the structure of an \(\FI\)-module, where the two actions commute. The restriction functor \(\Res: \FI \times G \to G \times S_n\) still has a left adjoint \(\Ind^{\FI \times G}: G \times S_n \to \FI \times G\) given by
\[ \Ind^{\FI \times G}(V)_{n+k} = \Ind^{S_{n+k}}_{S_n \times S_k} V \]
We can likewise define the category \(\FI\sharp \times G\). Then \(\Ind^{\FI \times G}(V)\) naturally have the structure of an \(\FI\sharp \times G\)-module, and the analogue of \thref{sharp_decomp} shows that any \(\FI\sharp \times G\)-module is a direct sum of these.

Notice that \(H^*(\Conf'_n(M/G); \mathbb{Q})\) is naturally an \(\FI \times G\)-module. Furthermore, if \(M\) satisfies the condition of \thref{boundary_sharp} then \(H_*(\Conf'_n(M/G); \mathbb{Q})\) is an \(\FI\sharp \times G\)-module.

\begin{prop}
\thlabel{prime_fg}
If \(M\) is a connected orientable manifold of dimension at least 2 with \(H^*(M; \mathbb{Q})\) of finite type, and \(G\) a finite abelian group acting freely and properly discontinuously on \(M\), then \(H^*(\Conf'(M/G); \mathbb{Q})\) is an \(\FI \times G\)-algebra of finitely-generated type.

If \(N\) is the interior of a connected compact manifold of dimension \(\ge 2\) with nonempty boundary, \(M \to N\) is a \(G\)-cover with \(G\) f.g. abelian, \(M\) is connected and orientable and \(H^*(M; \mathbb{Q})\) is a finite type \(G\)-module, then \(H_*(\Conf'(M/G); \mathbb{Q})\) is an \(\FI\sharp \times G\)-module of finitely-generated type.
\end{prop}
\begin{proof}
If \(G\) is finite this follows immediately by transfer applied to \(\Conf^G_n(M)\). However, even if \(G\) is infinite, we can still apply the Serre spectral sequence to the fibration
\[\Conf^G_n(M) \to \Conf'_n(M/G) \to BG^{n-1}\]
This spectral sequence has the form
\[ E_{i,j}^2 = H_i(BG^{n-1}; H_j(\Conf^G_n(M); \mathbb{Q})) \implies H_{i+j}(\Conf'_n(M/G); \mathbb{Q}) \]
By finite generation of \(H_j(\Conf^G_n(M); \mathbb{Q})\) as an \(\FI_G\)-module, each of the terms \(E_{i,j}^2\) is a finitely-generated \(\FI \times G\)-module, and thus the whole \(E^2\) page is of finitely-generated type. So the same argument at the end of \thref{confg_fg} thus shows that \(H_*(\Conf'_n(M/G); \mathbb{Q})\) is an \(\FI \times G\)-module of finitely-generated type.
\end{proof}
\thref{prime_fg} thus has the following consequence.
\begin{cor}
\thlabel{gstab}
Let \(M \to N\) be a \(G\)-cover with \(G\) f.g. abelian, and \(M\) a connected orientable manifold of dimension \(\ge 2\) with \(H^*(M; \mathbb{Q})\) of finite type. Suppose further that either \(G\) is finite, or \(N\) is the interior of a compact manifold with nonempty boundary. Then for any \(i \ge 0\), the sequence of \(G\)-modules \(H_i(\UConf'_n(N); \mathbb{C})\) are isomorphic for \(n \gg 0\).
\end{cor}
\begin{proof}
The \(S_n\)-cover \(\Conf'_n(M/G) \to \UConf'_n(M/G)\) gives a transfer isomorphism of \(G\)-modules:
\[H_i(\UConf'_n(M/G); \mathbb{Q}) \cong \left(H_i(\Conf'_n(M/G); \mathbb{Q})\right)_{S_n}\]

If \(G\) is finite then a finitely generated \(\FI \times G\)-module is still finitely generated when just considered an \(\FI\)-module. Thus \(H_i(\Conf'_n(M/G); \mathbb{Q})\) is a finitely generated \(\FI\)-module, and representation stability tells us that \(H_i(\Conf'_n(N); \mathbb{C})_{S_n}\) stabilize when \(n \gg 0\). By the transfer isomorphism, we conclude the corollary.

On the other hand, suppose \(G\) is infinite and \(N\) is the interior of a closed manifold with boundary. Then \(H_i(\Conf'(M/G); \mathbb{Q})\) is a finitely generated \(\FI\sharp \times G\)-module by \thref{prime_fg}. Therefore
\[ H_i(\Conf'(M/G); \mathbb{Q}) = \bigoplus_j \Ind^{\FI \times G}(W_j) \]
for some finite collection \(W_1, \dots W_k\) of representations of \(S_{i_1} \times G, \dots, S_{i_k} \times G\). By Frobenius reciprocity,
\[ \left(\Ind^{\FI \times G}(W_j)_n\right)_{S_n} \cong (W_j)_{S_{i_j}} \]
Thus once \(n\) is at least the degree of generation of \(H_i(\Conf'_n(M/G); \mathbb{Q})\), the coinvariants \(H_i(\UConf'_n(M/G); \mathbb{Q})\) will always be isomorphic as \(G\)-modules. The conclusion again follows, and in this case we can explicitly say that stability occurs once we reach the degree of generation of \(H_i(\Conf'_n(M/G); \mathbb{Q})\). This is not necessarily the same as the degree of generation of \(H_i(\Conf_n(M/G; \mathbb{Q})\), since the spectral sequence argument of \thref{prime_fg} makes non-constructive use of Noetherianity.
\end{proof}

Notice that the sequence of \(G\)-modules \(H_i(\UConf'_n(M/G); \mathbb{C})\) do not have any obvious maps between them, or obvious larger category-representation structure. It is only by lifting to the cover, noticing that we have representation stability, and then descending, that we obtain actual stability of \(G\)-modules. In this respect, \thref{gstab} mirrors the proof of homological stability for unordered configuration spaces of manifolds in \cite{CEF} (this result was originally proven in \cite{Chu}). On the other hand, the \(\FI_G\)-module structure of \(H_i(\Conf^G(M))\) is crucial for the specifics of \thref{gstab}. If we forgot about the \(\FI_G\)-module structure, and merely thought of \(H^i(\Conf^G(M); \mathbb{C})\) as an FI-module, finite generation as an FI-module would be enough to show that \(\UConf'(M/G))\) satisfies homological stability \emph{in dimension} (i.e., as vector spaces), but it would not show that it satisfies homological stability \emph{as \(G\)-modules}.

Notice also that if \(G\) is infinite, then we are really getting a strong result about \(G\)-modules. We do not need to make things nicer by passing to the Grothendieck group: we get honest stability of \(G\)-modules. This despite the fact that in general we can say very little about what all the finitely-generated \(G\)-modules actually look like.

\section{Complex reflection groups}
One classical example of a finite wreath product is the \emph{complex reflection group} \(G(d,1,n)\), where \(G = \mathbb{Z}/d\mathbb{Z}\) and so \(W_n = \mathbb{Z}/d\mathbb{Z} \wr S_n\), also sometimes referred to as the \emph{full monomial group}. Because of the action of \(G\) on \(\mathbb{C}\) as multiplication by \(d\)-th roots of unity, \(W_n\) acts on \(\mathbb{C}^n\) by generalized permutation matrices whose entries are \(d\)-th roots of unity. This gives rise to a rich class of examples of \(\FI_G\)-modules, which we now describe.

Note that since \(G\) is abelian, a splitting field for \(G\) is just the same as a field containing all the character values of \(G\). So in this case there is a minimal splitting field of characteristic zero, namely the field generated by the character values, which is \(\mathbb{Q}(\zeta_d)\). 

\subsection{Orbit configuration space of complex reflection groups}
First, of course, we can consider the orbit configuration space where \(M = \mathbb{C}^*\) and \(G = \mathbb{Z}/d\mathbb{Z}\) acting, as just described, as multiplication by \(d\)-th roots of unity. Thus
\[
\Conf^G_n(\mathbb{C}^*) = \left\{(v_i) \in \mathbb{C}^n \mid v_i \ne \zeta_d^k v_j \text{ for } i \ne j, v_i \ne 0 \text{ for all } i\right\}
\]
\(\Conf^G_n(\mathbb{C}^*)\) is thus the complement of the hyperplanes fixed by the standard complex-reflection generators of \(W_n\). This arrangement, called the \emph{complex reflection arrangement}, is much studied. For instance, Bannai \cite{Ba} proved that the complement is aspherical; its fundamental group is referred to as the \emph{pure monomial braid group}, and sometimes denoted \(P(d,n)\). Thus the cohomology of \(\Conf^G_n(\mathbb{C}^*)\) is isomorphic to the group cohomology of \(P(d,n)\).

Orlik-Solomon \cite{OS} calculated the cohomology of the complement of any hyperplane arrangement, as follows. Let \(\mathcal{A}\) be a collection of linear hyperplanes in \(\mathbb{C}^n\), and put \(M(\mathcal{A}) = \mathbb{C}^n \setminus \bigcup \mathcal{A}\). Say that a subset \(\{H_1, \dots, H_p\} \subset \mathcal{A}\) is \emph{dependent} if \(H_1 \cap \dots \cap H_p = H_1 \cap \cdots \widehat{H_i} \cdots \cap H_p\); alternately, if the linear forms defining the \(H_i\) are linearly dependent. Now let
\[
E(\mathcal{A}) = \bigwedge \langle e_H \mid H \in \mathcal{A}\rangle = \bigwedge H^1(M(\mathcal{A}); \mathbb{Q})
\]
and let \(I(\mathcal{A}) \subset E(\mathcal{A})\) be the ideal
\[
I(\mathcal{A}) = \left( \sum_{i = 1}^p (-1)^{i} e_{H_1} \wedge \cdots \wedge \widehat{e_{H_i}} \wedge \cdots \wedge e_{H_p}\,\middle|\, H_1, \dots, H_p \text{ are dependent}\right).
\]
Then \(H^*(M(\mathcal{A}); \mathbb{Q}) = E(\mathcal{A})/I(\mathcal{A})\). We then obtain the following.
\begin{thm}[\bfseries Cohomology of complex reflection arrangements]
\thlabel{crg_fg}
\(H^*(\Conf^{\mathbb{Z}/d\mathbb{Z}}_n(\mathbb{C}^*); \mathbb{Q}) = H^*(P(d,1,n); \mathbb{Q})\) is a finite type \(\FI_G\sharp\)-algebra. For each \(i\), the characters of the \(W_n\)-representations \(H^i(P(d,1,n); \mathbb{Q}(\zeta_d))\) are given by a single character polynomial of degree \(2i\) for all \(n\), and therefore \(H^i(P(d,1,n; \mathbb{Q}(\zeta_d))\) satisfies representation stability with stability degree \(\le 4i\).
\end{thm}
\begin{proof}
Since \(\mathbb{C}^*\) is the interior of an compact orientable 2-manifold with boundary, \thref{confg_fgsharp} says that \(H^*(\Conf^{\mathbb{Z}/d\mathbb{Z}}_n(\mathbb{C}^*); \mathbb{Q})\) is a finitely-generated \(\FI_G\sharp\)-module. To determine the degree this \(\FI_G\)-module is generated in, we use the Orlik-Solmon presentation. The hyperplane arrangement is:
\[\mathcal{A} = \{z_i \mid 1 \le i \le n\} \cup \{e_{i,a,j} \mid 1 \le i \ne j \le n; a \in G\}\]
where \(z_i\) is the hyperplane \(\{v_i = 0\}\) and \(e_{i,a,j}\) is the hyperplane \(\{v_i = \zeta^a v_j\}\), so that \(e_{i,a,j} = e_{j,a^{-1},i}\). To understand \(I(\mathcal{A})\), it is helpful to consider the embedding mentioned in \S3.1,
\[ \Conf^G_n(\mathbb{C}^*) \hookrightarrow \Conf_{|G|n}(\mathbb{C}),\;\; (v_1, \dots, v_n) \mapsto (v_1, \zeta v_1, \zeta^2 v_1, \dots, \zeta^{d-1} v_n) \]
This induces a map \(H^*(\Conf_{|G|n}(\mathbb{C}); \mathbb{Q}) \to H^*(\Conf^G_n(\mathbb{C}^*); \mathbb{Q})\). Orlik-Solomon likewise determines the cohomology of \(\Conf_{|G|n}(\mathbb{Q})\). The hyperplane arrangement for \(\Conf_{|G|n}(\mathbb{C})\) is
\[\mathcal{B} = \{f_{i, g; j, h} \mid 1 \le i,  j \le n;\, g, h \in G;\, (i, g) \ne (j, h)\} \]
The induced map \(H^1(\Conf_{|G|n}(\mathbb{C}); \mathbb{Q}) \to H^1(\Conf^G_n(\mathbb{C}^*); \mathbb{Q})\) is just given by
\[ f_{i,g; j, h} \mapsto \begin{cases} e_{i,h/g,j} & \text{if } i \ne j \\ z_{i} & \text{if } i = j \end{cases}\]
and is thus evidently surjective. Since \(H^*(\Conf^G_n(\mathbb{C}^*); \mathbb{Q})\) is generated by \(H^1\), this means that the total induced map \(H^*(\Conf_{|G|n}(\mathbb{C}); \mathbb{Q}) \to H^*(\Conf^G_n(\mathbb{C}^*); \mathbb{Q})\) is surjective. We can therefore describe the ideal \(I(\mathcal{A})\) in terms of the simple relations generating the ideal \(I(\mathcal{B})\) of the braid arrangement. We obtain the following presentation:
\[
H^*(M(\mathcal{A}); \mathbb{Q}) = \bigwedge\langle e_{i,a,j}, z_k \rangle \bigg/ \left\langle \begin{matrix} e_{i,a,j} \wedge e_{j,b,k} = (e_{i,a,j} - e_{j,b,k}) \wedge e_{i,ab,k} \\ z_i \wedge z_j = (z_i - z_j) \wedge e_{i,a,j} \\ e_{i,a,j} \wedge e_{i,b,j} = (e_{i,a,j} - e_{i,b,j})\wedge z_i \end{matrix} \right\rangle
\]
Another way to view these is by writing \(e_{i,a,i} = z_i\) for any \(a \ne 1 \in G\), as suggested by the induced map on \(H^1\) above: then the first relation, if \(i,j,k\) are allowed to equal one another, implies the others. As described in \S3.1, \(G^n \rtimes S_n\) acts on \(H^1\) as follows: on the \(\{z_i\}\), \(G^n\) acts trivially and \(S_n\) acts in the standard way, and on the \(\{e_{i,a,j}\}\) the action is:
\[ \left(\sigma, \zeta^{\vec{b}}\right) \cdot e_{i,a,j} = e_{\sigma(i),\,a - b_i + b_j,\,\sigma(j)} \]

So we see that \(H^*(\Conf^{\mathbb{Z}/d\mathbb{Z}}_n(\mathbb{C}^*); \mathbb{Q})\) is generated as an algebra by \(H^1(\Conf^{\mathbb{Z}/d\mathbb{Z}}_n(\mathbb{C}^*); \mathbb{Q})\), and that \(H^1\) is generated as an \(\FI_G\)-module by \(\{e_{1,a,2}\}\) and \(z_1\), and is therefore generated in degree \(2\). By \thref{alg}, \(H^i\) is generated in degree \(2i\), and so the stable range follows from \thref{sharp_range}.
\end{proof}

Wilson \cite[Thm 7.14]{Wi} proved the case \(d = 2\), for which the arrangement is the type \(BC\) braid arrangement. The decomposition for \(H^1\), following (\ref{irr_decom}), is
\[ H^1(P(d,1,n); \mathbb{Q}) = \Ind^{\FI_G}((1)_{\chi_0}) \oplus \bigoplus_{\chi \in \Irr(G)} \Ind^{\FI_G}((2)_{\chi}) \]
where \(\chi_0\) is the trivial character of \(G\). That is, \(\Ind^{\FI_G}((1)_{\chi_0})\) picks out the submodule spanned by the \(\{z_i\}\), and \(\bigoplus_{\chi \in \Irr(G)} \Ind^{\FI_G}((2)_{\chi})\) the submodule spanned by the \(\{e_{i,a,j}\}\). We can also compute the decomposition for \(H^2\):
\begin{gather*} 
H^2(P(d,1,n); \mathbb{Q}) = \Ind^{\FI_G}( (2,1)_{\chi_0}) \oplus \Ind^{\FI_G}( (3)_{\chi_0}) \oplus \bigoplus_{\chi \in \Irr(G)} \Ind^{\FI_G}( (3,1)_\chi) \oplus \Ind^{\FI_G}( (2)_\chi) \\
\oplus \bigoplus_{\substack{\chi \in \Irr(G) \\ \chi \ne \chi_0}} \Ind^{\FI_G}( (2)_\chi) \oplus \Ind^{\FI_G}((1)_{\chi_0}, (2)_\chi)^2 \oplus \Ind^{\FI_G}((1)_\chi, (1,1)_\chi) \oplus \Ind^{\FI_G}((2)_{\chi_0}, (2)_\chi)
\end{gather*}
These calculations agree with Wilson's \cite[p. 123]{Wi} for the case \(d = 2\).

\subsection{Diagonal coinvariant algebras}
Complex reflection groups also provide a generalization of the coinvariant algebra, as follows. Let \(k\) be a field with a primitive \(d\)-th root of unity \(\zeta_d\), and let \(V_n = k^n\) be the canonical representation of \(W_n\) by \(\zeta_d\)-power permutation matrices. Consider
\[
k[V_n^{\oplus r}] \cong k[x_1^{(1)}, \cdots, x_n^{(1)}, \cdots, x^{(r)}_1, \cdots x^{(r)}_n]
\]
This has a natural grading by \(r\)-tuples \((j_1, \dots, j_r)\), where \(j_i\) counts the total degree in the variables \(x_1^{(i)}, \dots, x_n^{(i)}\). There is an action of \(W_n\) induced by the diagonal action on \(V^{\oplus r}\) which respects the grading. Let \(\mathcal{I}_n = (k[V_n]^{W_n}_+)\) be the ideal generated by the constant-term-free \(W_n\)-invariant polynomials. The \emph{diagonal coinvariant algebra} is
\[\mathcal{C}^{(r)}(n) = k[V_n^{\oplus r}] / \mathcal{I}_n\]
These \(\mathcal{C}^{(r)}(n)\) arise naturally because, by the celebrated Chevalley-Shephard-Todd Theorem, complex reflection groups exactly comprise the class of groups \(G\) acting on a vector space \(V\) for which \(k[V]^G\) is a polynomial algebra. The representation theory of \(\mathcal{C}^{(r)}(n)\) has been much studied. The case \(r = 1\) is classical, going back to Chevalley's original paper \cite{Ch} proving the general case of Chevalley-Shephard-Todd.

The case \(r = 2\) for a complex reflection group (\(d > 2\)) was first studied by Haimain \cite{Ha} and Gordon \cite{Go}. The case of a general \(r\) was analyzed by Bergeron in \cite{Ber}. Bergeron shows that for a fixed \(n\), the multigraded Hilbert polynomial of \(\mathcal{C}^{(r)}(n)\) can be described independently of \(r\). (In a sense, this is orthogonal to our work, where \(r\) is fixed and \(n\) varies.) In general, very little is known about the \(\mathcal{C}^{(r)}(n)\) for \(n \ge 3\), even their dimension.
\begin{thm}[\bfseries Finite generation for diagonal coinvariant algebras]
The sequence of coinvariant algebras \(\mathcal{C}^{(r)}(n)\) form a graded co-\(\FI_{\mathbb{Z}/d\mathbb{Z}}\)-algebra of finite type.
\end{thm}
\cite[Thm 1.11]{CEF} proved the case \(d = 1\), and Wilson \cite[Thm 7.8]{Wi} proved it for \(d = 2\).
\begin{proof}
Let \(V\) be the \(\FI_G\)-module associated to the \(n\)-dimensional \(W_n\)-representations \(V_n = k^n\), where \(G\) acts by multiplication by \(\zeta_d\). Let \(\chi_1\) be the character of \(\mathbb{Z}/d\mathbb{Z}\) sending 1 to \(\zeta_d\), and let \(\underline{\lambda}\) be the \(G\)-partition of 1 supported on \(\chi_1\), so \(\underline{\lambda}(\chi_1) = (1)\). Then \(V_n = L(\underline{\lambda})_n \oplus L(\underline{\lambda})\). Thus \(k[V^{\oplus r}]\) is naturally a co-\(\FI_G\)-module of finite type. The ideals \(\mathcal{I}_n\) together form a co-\(\FI_G\)-ideal \(\mathcal{I}\), determined by the \(W_n\) action and the maps
\begin{align*}
(\iota_n)^*: \mathcal{I}_{n+1} &\to \mathcal{I}_n \\
x_i &\mapsto \begin{cases} x_i & i \le n \\ 0 & i = n+1 \end{cases}
\end{align*}
Thus \(\mathcal{C}^{(r)}\) is a co-\(\FI_G\)-quotient of a co-\(\FI_G\)-algebra of finite type, so it is a co-\(\FI_G\)-algebra of finite type. Since \(W_n\) preserves the grading of \(\mathcal{C}^{(r)}(n)\), then as a co-\(\FI_G\)-module, \(\mathcal{C}^{(r)}\) decomposes into graded pieces \(\mathcal{C}^{(r)}_J\).
\end{proof}
By Noetherianity, each graded piece \(\mathcal{C}^{(r)}_J\) is a finitely-generated co-\(\FI_G\)-module by \thref{alg}. So we obtain the following.
\begin{cor}
For each multi-index \(J\), the characters of the \(W_n\)-representations \(\mathcal{C}^{(r)}_J(n)\) are given by a single character polynomial for all \(n \gg 0\). 
\end{cor}

\subsection{Lie algebras associated to the lower central series}
For any group \(\Gamma\), recall that the \emph{lower central series} is defined inductively by \(\Gamma_1 = \Gamma\) and \(\Gamma_{k+1} = [\Gamma_k, \Gamma]\). The \emph{associated graded Lie algebra} \(\gr(\Gamma)\) over a field \(k\) is defined as
\[
\gr(\Gamma) = \bigoplus_{k \ge 1} (\Gamma_{k+1}/\Gamma_k) \otimes_{\mathbb{Z}} k
\]
where the Lie algebra structure is induced by the commutator. For example, if \(\Gamma\) is a free group of rank \(n\), then \(\gr(\Gamma)\) is the free lie algebra of rank \(n\). Since the higher \(\Gamma_i\) are generated by iterated commutators, \(\gr(\Gamma)\) is always generated as a Lie algebra by \(\gr(\Gamma)_1 = H_1(\Gamma;k)\).

Now, as \cite{CEF} carefully note, \(\pi_1\) is not a functor from topological spaces to groups, since it depends on a basepoint, and a map of spaces only determines a group homomorphism up to conjugacy. However, since conjugation acts trivially on the associated Lie algebra, there is a functor \(Z \mapsto \gr \pi_1(Z)\) from the category Top of topological spaces to the category of graded Lie algebras. Thus \(\gr \pi_1(\Conf^G(M))\) has the structure of a graded co-\(\FI_G\)-algebra. 

Here we continue to look at the case \(M = \mathbb{C}^*\), \(G = \mathbb{Z}/d\), so \(\pi_1 = P(d,n)\). The associated graded Lie algebra \(\gr P(d,n)\) has been studied, for example, by D. Cohen in \cite{Co}, who obtained some partial results, but still, little is known about it.
\begin{thm}[\bfseries Finite generation for \(\gr P(d,n)\)]
For each \(i\), the characters of the \(W_n\)-representations \(\gr P(d,n)^i\) are given by a single character polynomial for all \(n\).
\end{thm}
\begin{proof}
\(H_1(\Conf^G(\mathbb{C}^*)\) is a finitely generated \(\FI_G\sharp\)-module by \thref{crg_fg}. Since \(H_1(\Gamma)\) generates \(\gr(\Gamma)\) as a Lie algebra, this means that \(\gr P(d,n)\) is an \(\FI_G\sharp\)-algebra of finite type by \thref{alg}. Thus by Noetherianity, its graded pieces are finitely generated \(\FI_G\sharp\)-modules.
\end{proof}

\cite[Thm 7.3.4]{CEF} proved the case \(d = 1\), for which \(P(1,n) = P_n\), the pure braid group.

\section{Affine arrangements}
A motivating example of an \(\FI_G\)-module for \(G\) infinite, and one with interesting applications to affine braid groups, is the analogue of \(\Conf^{\mathbb{Z}/d}(\mathbb{C}^*)\) when we ``let \(d\) go to infinity'' and replace \(\mathbb{Z}/d\) with the infinite group \(\mathbb{Z}\).
Whereas the \(\mathbb{Z}/d\)-cover of \(\mathbb{C}^*\) is just homeomorphic to \(\mathbb{C}^*\) again, the \(\mathbb{Z}\)-cover of \(\mathbb{C}^*\) is homeomorphic to \(\mathbb{C}\). To be specific, we take \(M = \mathbb{C}\) and \(G = \langle \tau \rangle \cong \mathbb{Z}\) acting as translation by some \(\tau \in \mathbb{C}\). Thus \(M/G\) is homeomorphic to a cyclinder, which is homeomorphic to the punctured plane, and the quotient map is \(\exp(x/2\pi i \tau): \mathbb{C} \to \mathbb{C}^*\).

As in the case of complex reflection groups, \(\Conf_n^G(M)\) is the complement of the hyperplane arrangement
\[ \mathcal{A}_n = \{(z_i) \in \mathbb{C}^n \mid z_i - z_j \not\in \tau\mathbb{Z}\}\]
so one might hope that Orlik-Solomon gives the homology or cohomology. Unfortunately, \(\mathcal{A}_n\) is an infinite arrangement, and their arguments hinge on inducting on the number of hyperplanes (by looking at the ``deleted'' and ``restricted'' arrangements), so their theorem is not directly applicable.

However, a result of Cohen-Xicot\'encatl \cite{CX} provides a solution. They compute the homology and cohomology of a similar space---that of \(\Conf^{\mathbb{Z}^2}(\mathbb{C})\), mentioned earlier, so that the quotient \(\mathbb{C}/\mathbb{Z}^2\) is a torus---using a Fadell-Neuwirth type argument, that inducts on the dimension of the space, rather than the number of hyperplanes. 

In fact, the argument they give proves the following general result. Following Falk-Randell \cite{FR}, we first make the following definitions. Let \(\mathcal{A}\) be an arrangement of hyperplanes in \(\mathbb{C}^n\), and put \(M = \mathbb{C}^n \setminus \bigcup \mathcal{A}\). Say that \(M\) is \emph{strictly linearly fibered} if after a suitable linear change of coordinates the restriction to the first \(n-1\) coordinates is a fiber bundle projection with base \(N\) the complement of an arrangement in \(\mathbb{C}^{n-1}\) and fiber a complex line with a discrete set of points removed. Now define a \emph{fiber-type} arrangement inductively as follows: \begin{itemize}
\item The 1-arrangement \(\{0\} \subset \mathbb{C}\) is \emph{fiber-type}
\item An \(n\)-arrangement for \(n \ge 2\) is \emph{fiber-type} if the complement \(M\) is strictly linearly fibered over \(N\), and \(N\) is the complement of an \((l-1)\)-arrangement of fiber type.
\end{itemize}
Thus the complement \(M\) of a fiber-type \(n\)-arrangement sits atop a tower of fibrations
\begin{equation} \label{fiber-type}
M = M_n \xrightarrow{p_n} M_{n-1} \xrightarrow{p_{n-1}} \cdots \to M_2 \xrightarrow{p_2} M_1 = \mathbb{C}^*
\end{equation}
with the fiber \(F_k\) of \(p_k\) homeomorphic to \(\mathbb{C}\) with a discrete set of points removed.

In particular, Falk-Randell use the tower of fibrations (\ref{fiber-type}) to prove that \(M\) is aspherical, that each bundle map \(p_k\) has a section, and that \(\pi_1(M_{k-1})\) acts trivially on \(H_i(F_k)\).

\begin{prop}[{\cite[Prop 6]{CX}}]
Let \(\mathcal{A}\) be an arrangement of hyperplanes in \(\mathbb{C}^n\) (possibly countably infinite and affine) which is discrete in the space of all hyperplanes. Put \(M = \mathbb{C}^n \setminus \bigcup \mathcal{A}\), and suppose \(\mathcal{A}\) is fiber-type.

Then \(H^*(M; \mathbb{Q})\) contains a dense subalgebra \(A\) generated by classes \(E_{H} \in H^1(M; \mathbb{Q})\) for each \(H \in \mathcal{A}\), which is isomorphic to the Orlik-Solomon algebra of \(\mathcal{A}\). Moreover, there is an isomorphism \(H_i(M; \mathbb{Q}) \cong A^i\).
\end{prop}

Cohen-Xicot\'encatl use the tower of fibrations (\ref{fiber-type}) to obtain this result. In a sense, this result was basically implicit in Falk-Randell \cite{FR}, but there the arrangements were always assumed to be finite.

The algebra \(A\) is the algebra one gets by following the recipe of the Orlik-Solomon algebra of \(\mathcal{A}\), even though \(\mathcal{A}\) happens to be infinite, with each generator \(E_H\) coming from monodromy around the hyperplane \(H\). So the situation is really as good as one could hope for.

\(\Conf^\mathbb{Z}(\mathbb{C})\) is closely related to affine Coxeter and Artin groups. The affine Coxeter group \(W_{\widetilde{A}_n}\) is the one associated to the Coxeter diagram \(\widetilde{A}_n\), which is obtained by adding one extra node connecting the first and last nodes to the Coxeter diagram \(A_n\). Concretely, let \(V^\mathbb{R} = \{(x_i) \in \mathbb{R}^n \mid \sum x_i = 0\}\). Then \(W_{\widetilde{A}_{n-1}}\) is the group generated by reflections in \(V^\mathbb{R}\) through the hyperplanes \(H^\mathbb{R}_s = \{x_i - x_j = n \mid n \in \mathbb{Z}\}\). This group \(W_{\widetilde{A}_{n-1}}\) is isomorphic to \(\mathbb{Z}^{n-1} \rtimes S_n\). Let \(V\) be the complexification of \(V^\mathbb{R}\), and let \(V_0 = V \setminus \bigcup H_s\) the hyperplane complement. Then \(\pi_1(V_0)\) is the pure Artin braid group \(P_{\widetilde{A}_{n-1}}\), the group \(W_{\widetilde{A}_{n-1}}\) acts freely on \(V_0\) and \(\pi_1(V_0/W_{\widetilde{A}_{n-1}})\) is the Artin braid group \(B_{\widetilde{A}_{n-1}}\). But notice that \(V_0\) is just the fiber of the bundle \(\Conf^\mathbb{Z}(\mathbb{C}) \to \mathbb{C}\) given by taking the sum of the coordinates, which is trivial since \(\mathbb{C}\) is contractible, so that \(V_0 \simeq \Conf^\mathbb{Z}(\mathbb{C})\). In fact, it has been proven \cite{Ok} that \(V_0\) is aspherical, and therefore \(\Conf^\mathbb{Z}(\mathbb{C})\) is a \(K(P_{\widetilde{A}_{n-1}}, 1)\).

We then obtain the following.

\begin{thm}[\bfseries Homology of type \(\widetilde{A}_n\) pure braid group]
\(H_*(\Conf^{\mathbb{Z}}(\mathbb{C}); \mathbb{Q}) = H_*(P_{\widetilde{A}_{n-1}}; \mathbb{Q})\) is a finitely-generated type \(\FI_G \sharp\)-module. For each \(i\), the \(W_n\)-representations \(H_i(P_{\widetilde{A}_{n-1}}; \mathbb{Q})\) satisfies \(K_0\)-stability with stability degree \(\le 4i\).
\end{thm}
\begin{proof}
By \thref{confg_fgsharp}, \(H_*(\Conf^\mathbb{Z}(\mathbb{C}); \mathbb{Q})\) is a finitely-generated type \(\FI_G\sharp\)-module, while \(H^*(\Conf^\mathbb{Z}(\mathbb{C}); \mathbb{Q})\) is an \(\FI_G\sharp\) algebra not of finitely-generated type.

Taking the usual coordinates on \(\mathbb{C}^n\), we see that \(\Conf^\mathbb{Z}(\mathbb{C})\) is a fiber-type arrangement. Proposition 5.1 therefore applies, and we obtain a presentation for \(H_*(\Conf^\mathbb{Z}(\mathbb{C}); \mathbb{Q})\).

Recall \cite{OT} that for affine arrangements, the Orlik-Solomon ideal is generated not just by \(\partial e_S\) for \(S\) dependent, but also by \(e_S\) when \(\cap S = \emptyset\) (possible for non-central arrangements), which is derived by considering the homogeneized arrangement. In the case of \(\Conf^\mathbb{Z}(\mathbb{C})\), \(\mathcal{A}\) consists of the affine hyperplanes \(H_{i,a,j} = \{z_i - z_j = a \mid a \in \mathbb{Z}\}\), and thus
\[
A = \bigwedge\langle e_{i,a,j} \rangle \bigg/ \left\langle \begin{matrix} e_{i,a,j} \wedge e_{j,b,k} = (e_{i,a,j} - e_{j,b,k}) \wedge e_{i,a+b,k} \\ e_{i,a,j} \wedge e_{i,b,j} = 0 \end{matrix} \right\rangle
\]
and \(H_i(\Conf^\mathbb{Z}(\mathbb{C}); \mathbb{Q}) \cong A^i\). As an algebra, \(A\) is generated by \(A^1\), which is generated as an \(\FI_G\)-module in degree 2, and so by \thref{alg}, \(A^i\) is generated in degree \(2i\). So by Theorem 2.5, \(H_i(\Conf^\mathbb{Z}(\mathbb{C}); \mathbb{Q})\) satisfies \(K_0\)-stability with stability degree \(4i\).

\end{proof}
Furthermore, since the Coxeter group \(W_{\widetilde{A}_{n-1}} \cong \mathbb{Z}^{n-1} \rtimes S_n\), then \(V_0 / W \simeq \UConf'_n(\mathbb{C}^*)\), in the notation of \S3.1, and therefore \(\UConf'_n(\mathbb{C}^*)\) is a classifying space for the full Artin group \(B_{\widetilde{A}_{n-1}}\). Here, Callegaro-Moroni-Salvetti \cite[Thm 4.2]{CMS} computed that
\begin{align*}
H_i(B_{\widetilde{A}_{n-1}}; \mathbb{Q}) = \begin{cases} \mathbb{Q}[t^{\pm 1}]/(1 + t) & 1 \le i \le n - 2 \\ \mathbb{Q}[t^{\pm 1}]/(1 + t) & \text{if } i = n-1, n \text{ odd} \\ \mathbb{Q}[t^{\pm 1}]/(1 - t^2) & \text{if } i = n-1, n \text{ even} \end{cases}
\end{align*}
as \(\mathbb{Q}[G] = \mathbb{Q}[t^{\pm 1}]\)-modules. This verifies the stabilization as \(\mathbb{Q}[G]\)-modules for \(n \gg 0\) proven in Corolarry 3.5, and in fact stabilization occurs once \(n \ge i + 2\).

These spaces also have close connections to the Burau representation. First, if we let
\[\UConf_{n, 1}(\mathbb{C}) = \{(\{x_1, \dots, x_n\}; x) \mid x_i \ne x_j, x_i \ne x\}\] then there is a bundle 
\[\UConf_n(\mathbb{C}^*) \to \UConf_{n,1}(\mathbb{C}) \to \mathbb{C}.\]
Since \(\mathbb{C}\) is contractible, \(\UConf_n(\mathbb{C}^*) \simeq \UConf_{n,1}(\mathbb{C})\). Similarly, let
\[\UConf'_n(\mathbb{C}^*) = \{(\{x_1, \dots, x_n\}; y) \mid x_i \ne 0, x_i \ne x_j, e^y = x_1 \cdots x_n\}\]
and let
\[\UConf'_{n, 1} = \{(\{x_1, \dots, x_n\}; x, y) \mid x_i \ne x, x_i \ne x_j, e^y = (x_1 - x) \cdots (x_n - x)\}.\]
Then \(\UConf'_n(\mathbb{C}^*) \cong \UConf'_{n, 1}\). Finally, notice that the inclusion \(\UConf_n(\mathbb{C}^*) \hookrightarrow \UConf_n(\mathbb{C})\) is homotopy equivalent to the bundle projection \(\UConf(\mathbb{C})_{n,1} \twoheadrightarrow \UConf(\mathbb{C})\) onto the first \(n\) coordinates. We therefore obtain the following diagram:
\begin{equation*}
\begin{diagram}
\UConf'_n(\mathbb{C}^*) && \simeq && \UConf'_{n,1}(\mathbb{C}) \\
\dTo &&&& \dTo \\
\UConf_n(\mathbb{C}^*) && \simeq && \UConf_{n,1}(\mathbb{C}) \\
& \rdTo && \ldTo & \\
&& \UConf_n(\mathbb{C}) &&
\end{diagram}
\end{equation*}
Thus \(\pi_1(\UConf_n(\mathbb{C})) = B_n\) acts up to conjugacy on the homotopy fiber of the map from \(\UConf'_n(C^*)\), which is just the actual fiber of the map from \(\UConf'_{n,1}(\mathbb{C})\). Call this fiber \(X_f\) for \(f \in \UConf_n(\mathbb{C})\). This gives an honest action of \(B_n\) on \(H_1(X_f; \mathbb{C})\), which is the Burau representation. Furthermore, replacing the group \(G = \mathbb{Z}\) with \(\mathbb{Z}/d\mathbb{Z}\), this construction gives the Burau representation specialized at \(\zeta_d = e^{2\pi i/d}\). Chen has explored this connection and its relation to point-counts of superelliptic curves in \cite{WCh}; in his notation, with \(G = \mathbb{Z}/d\mathbb{Z}\), \(\UConf'_{n,1}(\mathbb{C})\) is called \(E_{n,d}\). 

Finally, we likewise obtain results on the pure braid group of type \(\widetilde{C}_n\).
\begin{thm}[\bfseries Homology of the type \(\widetilde{C}_n\) pure braid group]
\(H_*(P_{\widetilde{C}_n}; \mathbb{Q})\) is a finitely-generated type \(\FI_G\) module, for \(G = \mathbb{Z} \rtimes \mathbb{Z}/2\). For each \(i\), the \(W_n\)-representation \(H_i(P_{\widetilde{C}_n}; \mathbb{Q})\) satisfies \(K_0\)-stability with stability degree \(\le 4i\).
\end{thm}
\begin{proof}
As \cite[\S4.3]{Mor} explains, the hyperplane complement for type \(\widetilde{C}_n\) is the following:
\[ Y_n = \{x \in \mathbb{C}^n \mid x_i \pm x_j \notin \mathbb{Z}, x_i \notin \tfrac{1}{2}\mathbb{Z}\} \]
Applying the map \(z \mapsto \exp(2\pi i z)\) componentwise gives a normal covering map \(Y_n \to Y'_n\) with deck group \(\mathbb{Z}^n\), where
\[ Y'_n = \{x \in (\mathbb{C}^*)^n \mid x_i \ne x_j^{\pm 1}, x_i \ne \pm 1\} \]
The map \(z \mapsto \frac{1-z}{1+z}\) is a self-homeomorphism of \(\mathbb{C} \setminus \{0,1,-1\}\) taking \(1/z\) to \(-z\). Applying this map componentwise to \(Y'_n\) therefore gives a homeomorphism
\[ Y'_n \cong \{x \in \mathbb{C}^n \mid x_i \ne \pm x_j, x_i \ne 0, 1, -1\} \]
Finally, applying the map \(z \mapsto z^2\) componentwise gives a normal covering map \(Y'_n \to \Conf_n(\mathbb{C} \setminus \{0,1\})\), with deck group \((\mathbb{Z}/2)^n\). In conclusion, if we write \(M = \mathbb{C} \setminus \frac{1}{2}\mathbb{Z}\), there is a normal cover \(M \to \mathbb{C} \setminus \{0,1\}\) given by \(z \mapsto \left( \dfrac{1 - e^{2\pi i z}}{1 + e^{2\pi i z}}\right)^2\) with deck group \(G = \mathbb{Z} \rtimes \mathbb{Z}/2\), and \(Y_n\) is homeomorphic to the orbit configuration space \(\Conf^G_n(M)\). 

By \thref{confg_fgsharp}, \(H_*(\Conf^G(M); \mathbb{Q})\) is a finitely-generated type \(\FI_G\sharp\)-module. Again, we use Prop 5.1 to obtain a presentation. The hyperplane arrangement consists of the hyperplanes \(E_{i,a,j} = \{z_i - z_j = a \mid a \in \mathbb{Z}\}\), \(H_{i,a,j} = \{z_i + z_j = a \mid a \in \mathbb{Z}\}\), and \(Z_{i,a} = \{2z_i \ne a \mid a \in \mathbb{Z}\}\). The Orlik-Solomon algebra is given by:
\[ A = \bigwedge \langle e_{i,a,j}, h_{i,a,j}, z_{i,b} \rangle \bigg/ \left\langle \begin{matrix} e_{i,a,j} \wedge e_{j,b,k} = (e_{i,a,j} - e_{j,b,k}) \wedge e_{i,a+b,k}; \;\; e_{i,a,j} \wedge e_{i,b,j} = 0 \\ 
e_{i,a,j} \wedge h_{j,b,k} = (e_{i,a,j} - h_{j,b,k}) \wedge h_{i,a+b,k}; \; \; h_{i,a,j} \wedge h_{i,b,j} = 0  \\
e_{i,a,j} \wedge h_{i,b,j} = (e_{i,a,j} - h_{i,b,j}) \wedge z_{i, a+b}; \; \; z_{i,a} \wedge z_{i,b} = 0 \\
z_{i,a} \wedge z_{j,b} = (z_{i,a} - z_{j,a}) \wedge h_{i, \frac{a+b}{2}, j} \text{ if } 2 \mid a + b, 0 \text{ otherwise}
\end{matrix} \right\rangle \]
and \(H_i(\Conf^G(M); \mathbb{Q}) \cong A^i\). In particular, we see that this \(\FI_G\sharp\)-module is generated in degree \(2i\). So by Theorem 2.5, \(H_i(\Conf^G(M); \mathbb{Q})\) satisfies \(K_0\)-stability with stability degree \(4i\).

\end{proof}
In particular, we can again use the calculation (\ref{irr_decom}) to write \(H_1\) as an induced module:
\begin{equation} \label{affC}
H_1(P_{\widetilde{C}_{n-1}}; \mathbb{Q}) = \Ind_1^{\FI_{G}} k[\mathbb{Z}] \oplus \Ind_2^{\FI_{G}} (V \oplus W)
\end{equation}
where \(V\) and \(W\) are both isomorphic as vector spaces to \(k[\mathbb{Z}]\), but as representations of \(k[\mathbb{Z} \times \mathbb{Z}]\), \(V\) has one copy of \(\mathbb{Z}\) acting positively and one acting negatively, whereas \(W\) has both acting positively. In this case, the image of \(V\) and \(W\). In this case, \([V] = [W] = 0\) in \(K_0(\mathbb{Z} \times \mathbb{Z})\), so we as expected we lose a lot of information in passing to \(K_0\). However, the first summand in (\ref{affC}) is retained in \(K_0\). To wit, the representation \(U = k[\mathbb{Z}]\) of \(\mathbb{Z} \rtimes \mathbb{Z}\) is just \(\Ind_{\mathbb{Z}/2}^{\mathbb{Z} \rtimes \mathbb{Z}/2} k\), so by Moody's induction theorem, this is nonzero in \(K_0(\mathbb{Z} \rtimes \mathbb{Z}/2)\). We therefore obtain, as we saw in the introduction,
\[ [H_1(P_{\widetilde{A}_{n-1}}; \mathbb{Q})] = [\Ind_{W_1 \times W_{n-1}}^{W_n} U] = [L((n)_U)] \]

\section{Automorphism groups of free products}
Given a group \(G\), let \(G^{*n} = G \ast \cdots \ast G\) be the \(n\)-fold free product. The group \(\Aut(G^{*n})\) contains a copy of \(\Aut(G) \wr S_n\), acting by automorphisms on each factor and permuting them. This normalizes the \emph{Fouxe-Rabinovitch} group \(\FR(G^{*n})\), the subgroup of \(\Aut(G^{*n})\) generated by partial conjugations, which conjugate the \(i\)-th free factor by the \(j\)-th free factor for some \(i \ne j\), and leaves fixed all factors besides the \(i\)-th.

Define the \emph{symmetric automorphism group} \(\Sigma\Aut(G^{*n})\) by
\[\Sigma\Aut(G^{*n}) = \FR(G^{*n}) \rtimes (\Aut(G)\,\wr\,S_n)\]
When \(G\) is freely indecomposable and not isomorphic to \(\mathbb{Z}\), then \(\Sigma\Aut(G^{*n})\) is in fact all of \(\Aut(G^{*n})\).

These groups have interesting geometric interpretations. For example, if \(G = \mathbb{Z}\) then \(\Sigma\Aut(G^{*n})\) is the fundamental group of \(n\) unknotted circles embedded in \(\mathbb{R}^3\), or the \emph{group of string motions}, and \(\FR(G^{*n})\) is the corresponding group of \emph{pure} string motions. More generally, Griffin \cite{Gr} constructed a classifying space for \(\FR(G^{*n})\), which he calls the \emph{moduli space of cactus products}.

The cohomology of these groups has also been much studied. For example, Hatcher-Wahl \cite{HW} and Collinet-Djament-Griffin \cite{CDG} proved that \(\Sigma\Aut(G^{*n})\) satisfy homological stability. On the other hand, the cohomology of \(\FR(G^{*n})\) does not stabilize: e.g., \(H^1(\FR(G^{*n})) = H^1(G)^{\oplus n(n-1)}\). 

However, notice that the exact sequence
\[1 \to \FR(G^{*n}) \to \Sigma\Aut(G^{*n}) \to \Aut(G) \wr S_n \to 1\]
gives an action of \(\Aut(G) \wr S_n\) on \(H^i(\FR(G^{*n}); \mathbb{Q})\). In fact, inner automorphisms of \(G\) act trivially on cohomology, so in fact we get a representation of \(\Out(G) \wr S_n\) on \(H^i(\FR(G^{*n}); \mathbb{Q})\). Wilson \cite[Thm 7.3]{Wi} proved that \(H^*(\FR(\mathbb{Z}^{*n}))\) is a finite type  \(\FI_{\mathbb{Z}/2}\)-algebra, since \(\Out(\mathbb{Z}) = \mathbb{Z}/2\).
\begin{thm}[\bfseries Finite generation for cohomology of Fouxe-Rabinovitch groups]
Suppose that \(G\) is any group such that \(H^*(G; k)\) is a finite type \(k\)-module. Then \(H^*(\FR(G^{*n}); k)\) is a finite type \(\FI_{\Out G}\sharp\)-algebra.
\end{thm}
Notice that, for example, there are many groups with even no nontrivial outer automorphisms, for which Theorem 6.1 tells us that \(H^*(\FR(G^{*n}); \mathbb{Q})\) is a finite type \(\FI\sharp\)-algebra.

\begin{proof}
We first need to show that \(H^*(\FR(G^{*n}))\) is, in the first place, an \(\FI_{\Out(G)}\sharp\)-algebra. Griffin \cite[Thm B]{Gr} proved that \(H^*(\FR(G^{*n})) \cong H^*((G^{*n})^{n-1})\) as an algebra, so \(H^*(\FR(G^{*n}))\) is generated by \(\widetilde{H}^*(G)^{\oplus n(n-1)} \cong \bigoplus_{i \ne j} \widetilde{H}^*(G)\). Under this isomorphism, for any \(\alpha \in \widetilde{H}^*(G)\), write \(\alpha_{i,j}\) for the element \(\alpha_{i,j} \in \bigoplus_{i \ne j} \widetilde{H}^*(G)\) sitting in the \((i,j)\)-th summand. The map induced by an \(\FI_{\Out(G)}\sharp\) morphism \((s,\vec \phi): [m] \to [n]\) is:
\begin{align*} (s,\vec \phi)_*: \bigoplus_{i \ne j} \widetilde{H}^*(G) &\to \bigoplus_{i \ne j} \widetilde{H}^*(G) \\
\alpha_{i,j} & \mapsto \begin{cases} \phi^*_j\,\alpha_{s(i), s(j)} & \text{if } i, j \in \operatorname{domain}(s) \\
0 & \text{otherwise} \end{cases}
\end{align*}
Therefore  \(V_n = \bigoplus_{i \ne j} \widetilde{H}^*(G)_{i,j}\) has the structure of an \(\FI_{\Out(G)}\sharp\)-module. \(V\) is generated in finite degree, since it is just generated by \(\widetilde{H}^*(G)_{1,2}\), and each \(V^i_n\) is finite-dimensional, since by assumption each \(H^j(G)\) is, so therefore \(V\) is finitely generated. Furthermore, the action of \(\FI_{\Out(G)}\sharp\) on \(V\) extends to an algebra map of \(H^*(\FR(G^{*n}))\), which is thus an \(\FI_{\Out(G)}\sharp\)-algebra. Since it is generated by the finite type module \(V\), then \(H^*(\FR(G^{*n}))\) is finite type by \thref{alg}.
\end{proof}
\begin{cor}[\bfseries Representation stability for cohomology of Fouxe-Rabinovitch groups]
Suppose that \(G\) is any group with \(H^*(G; \mathbb{Q})\) of finite type and \(\Out G\) finite. Then for each \(i\), the characters of the \(\Out(G) \wr S_n\)-representations \(H^i(\FR(G^{*n}); \mathbb{C})\) are given by a single character polynomial of degree \(2i\) for all \(n\), and therefore \(H^i(\FR(G^{*n}); \mathbb{C})\) is representation stable, with stability degree \(4i\).
\end{cor}
\begin{proof}
As shown in the proof of Theorem 7.1, \(H^*(\FR(G^{*n})\) is generated as an algebra by \(V\), which is generated as an \(\FI_G\)-module by \(\widetilde{H}^*(G)_{1,2}\), and so \(V\) is generated in degree 2. By \thref{alg}, \(H^i\) is generated in degree \(2i\), and so the stable range follows from \thref{sharp_range}.
\end{proof}
\bibliography{general}

\end{document}